\documentclass[10pt]{amsart}

\usepackage{latexsym}
\usepackage{amssymb}
\usepackage{amsmath}
\usepackage{xcolor}
\usepackage{tikz-cd}
\usepackage{mathrsfs}
\usepackage{mathtools}

\usepackage[italian, english]{babel}
\usepackage{url}

\DeclareMathAlphabet{\mathpzc}{OT1}{pzc}{m}{it}

\newtheorem{theorem}{Theorem}[section]
\newtheorem{lemma}[theorem]{Lemma}
\newtheorem{proposition}[theorem]{Proposition}
\newtheorem{corollary}[theorem]{Corollary}

\theoremstyle{definition}
\newtheorem{definition}[theorem]{Definition}

\newtheorem{example}[theorem]{Example}

\theoremstyle{remark}

\newtheorem{remark}[theorem]{Remark}
\newtheorem{fact}[theorem]{Fact}

\newtheorem{notation}[theorem]{Notation}

\def\N{\mathbb{N}}

\def\R{\mathbb{R}}

\def\PP{\mathcal{P}}
\def\U{\mathcal{U}}
\def\V{\mathcal{V}}
\def\W{\mathcal{W}}

\def\Xc{\mathcal{X}}
\def\Yc{\mathcal{Y}}
\def\Zc{\mathcal{Z}}

\makeindex

\begin{document}

\title{The magic of tensor products of ultrafilters}

\author{Mauro Di Nasso}






\begin{abstract}
Tensor products of ultrafilters have special combinatorial features 
closely related to Ramsey's Theorem, making them useful tools in applications. 
Here we first review their fundamental properties and isolate some new ones, 
including a characterisation of the limit superior and inferior of sequences 
as limits along tensor products, and an application to the Banach asymptotic density. 
We then prove a general result on the combinatorial structure of sets 
belonging to tensor products and, as a result, we obtain several characterisations 
of the additive properties of sets of natural numbers.
Finally, we show that tensor products can be described as idempotents 
of an appropriate semigroup.
\end{abstract}

\maketitle

\tableofcontents

\section*{Introduction}

Ultrafilters are special mathematical objects that have been 
extensively studied both as such and for their many applications 
in different areas of mathematics, including topology, Banach spaces, Ramsey Theory and combinatorics. 
Ultrafilters can be seen as finitely additive probability measures defined on
all subsets of a given set, with the special feature that they only takes
the two values $0$ or $1$.
Here we focus on their tensor products, a canonical way of producing 
ultrafilters on Cartesian products that can be seen as 
the products of the corresponding measures.
Tensor products of ultrafilters have special combinatorial properties 
closely related to Ramsey's Theorem that make them useful tools in the practice. 
In this paper we review their fundamental properties and isolate some new ones.
In particular, we prove a general result on the combinatorial structure of sets 
belonging to tensor products.

\smallskip
The paper is organized as follows.
After a first section where we recall the basic notions about ultrafilters
and their tensor products,
in the second section we make explicit the connection
between tensor products and Ramsey's Theorem.
Sections 3 and 4 are dedicated to simple characterisations
of limits of double-indexed sequences
and of limit superior and limit inferior of sequences, as limits along tensor products; 
an application to Banach asymptotic density is then presented. 
In the following Section 5 we characterise tensor products as the idempotent elements
of the semigroup on the space of ultrafilters, as originated
from the simple associative operation on pairs: $(a,b)\star(a',b')=(a,b')$.
In Section 6 we present a general characterisation of the sets 
belonging to tensor products in terms of their Ramsey properties,
and we provide several examples. In particular, we show how such characterisations
allow to reformulate in a convenient and compact way
the properties of containing certain sum-sets,
as considered in additive combinatorics.
The last section contains the proof of the main result presented in the previous section.

\medskip
\section{Preliminary notions and facts}

Let us recall the basic notions we will use throughout the paper.

\smallskip
Let $I$ be an arbitrary nonempty set.
Recall that an \emph{ultrafilter} on $I$ is a family $\U$ of subsets of $I$
that satisfies the following properties:

\begin{enumerate}
\item[(U1)]
$\emptyset\notin\U$ and $I\in\U$.
\item[(U2)]
If $A\supseteq B$ are subsets of $I$ and $B\in\U$ then $A\in\U$.
\item[(U3)]
If $A, B\in\U$ then $A\cap B\in\U$.
\item[(U4)]
If $A$ is a subset of $I$ and $A\notin\U$ then the complement $A^c=I\setminus A\in\U$.
\end{enumerate}

An alternative presentation of ultrafilters is by means of
``special" measures defined on all subsets.

\begin{fact}
Given an ultrafilter $\U$ on $I$, the function $\mu_\U:\PP(I)\to\{0,1\}$
where $\mu_\U(A)=0\Leftrightarrow A\in\U$ and $\mu_\U(A)=1\Leftrightarrow A\in\U$,
is a finitely additive measure. Conversely, given
a finitely additive two-valued measure $\mu:\PP(I)\to\{0,1\}$ defined on all subsets of $I$,
the family $\U_\mu=\{A\subseteq I\mid \mu(A)=1\}$
of the sets of full measure is an ultrafilter on $I$.
Moreover, the two correspondences $\U\mapsto\mu_\U$ and $\mu\mapsto\U_\mu$
are one the inverse of the other, \emph{i.e.} $\U_{\mu_\U}=\U$ for every ultrafilter $\U$ on $I$,
and $\mu_{\U_\mu}=\mu$ for every finitely additive two-valued measure $\mu:\PP(I)\to\{0,1\}$.
\end{fact}

Trivial examples are given by the \emph{principal ultrafilters} $\U_i=\{A\subseteq I\mid i\in I\}$
generated by an element $i\in I$. It is easily seen that an ultrafilter
$\U$ on $I$ is non-principal if and only if it includes the family
of \emph{cofinite subsets} of $I$, \emph{i.e.} $A\in\U$ whenever the complement $I\setminus A$ is finite.

\smallskip
Given an ultrafilter $\U$ on a set $I$ and a function $f:I\to J$,
one obtains an ultrafilter $f(\U)$ on $J$, called the \emph{image ultrafilter} of $\U$
under $f$, by setting for every $X\subseteq J$:
$$X\in f(\U)\Leftrightarrow f^{-1}(X)\in\U.$$

A canonical operation between ultrafilters, on which this paper is focused,  is the following.

\begin{definition}
Let $\U$ and $\V$ be ultrafilters on the sets $I$ and $J$, respectively.
Their \emph{tensor product} (sometimes called \emph{Fubini product}) is the ultrafilter $\U\otimes\V$
on the Cartesian product $I\times J$ defined by setting for every $X\subseteq I\times J$:
$$X\in\U\otimes\V\ \Longleftrightarrow\
\{i\in I\mid X_i\in\V\}\in\U$$
where $X_i:=\{j\in I\mid (i,j)\in X\}$ is the vertical $i$-fiber of $X$.
\end{definition}

It can be directly verified from the definition that $\U\otimes\V$ is actually an ultrafilter.
It is also easily verified that $\U\otimes\V$ extends the \emph{product family} $\U\times\V$:
$$\U\times\V:=\{A\times B\mid A,\in\U, B\in\V\}.$$

Note that the product family $\U\times\V$ has the finite intersection property but it is not
a filter because it is not closed under supersets.

\begin{fact}
The tensor product operation between ultrafilters is associative.
Indeed, given ultrafilters $\U,\V,\W$ on the sets $I$, $J$ and $K$ respectively,
it can be directly verified from the definitions that
$\U\otimes(\V\otimes\W)=(\U\otimes\V)\otimes\W$,
provided one identifies the Cartesian products
$I\times(J\times K)$ and $(I\times J)\times K$.
\end{fact}

\begin{fact}
The tensor product of ultrafilters corresponds to
the product measure. Indeed, given $\U,\V$ ultrafilters on $I$ and $J$ respectively,
it is then easily seen that 
$\U\otimes\V=\{X\subseteq I\times J\mid (\mu_\U\otimes\mu_\V)(X)=1\}$,
where $\mu_\U,\mu_\V:\PP(I)\to\{0,1\}$ are the finitely additive measures
corresponding to $\U$ and $\V$, respectively,
and $\mu_\U\otimes\mu_\V$ is their product measure on $I\times J$.
\end{fact}

Listed below are some
equivalent reformulation of the property of being a tensor product.

\begin{theorem}\label{tensorproductsequivalences}
Let $\W$ be an ultrafilter on $I\times J$. Then the following are equivalent:
\begin{enumerate}
\item[$(i)$]
$\W$ is a tensor product, \emph{i.e.} $\W=\U\otimes\V$
for suitable ultrafilters $\U$ and $\V$ on $I$ and $J$, respectively.
\item[$(ii)$]
$\W=\pi_1(\W)\otimes\pi_2(\W)$ is the tensor product of the image ultrafilters
under the canonical projections $\pi_1:I\times J\to I$ and $\pi_2:I\times J\to J$.
\item[$(iii)$]
For every $X\subseteq I\times J$, if $X_i\in\pi_2(\W)$ for all $i\in I$
then $X\in\W$.
\item[$(iv)$]
For every $X\subseteq I\times J$, if $X\in\W$ then there exists $i\in I$
with $X_i\in\pi_2(\W)$.
\end{enumerate}
Besides, if $\pi_1(\W)$ in not principal, then the above properties are also equivalent to:
\begin{enumerate}
\item[$(v)$]
For every $X\subseteq I\times J$, if $X_i\in\pi_2(\W)$ for all but finitely many $i\in I$
then $X\in\W$.
\item[$(vi)$]
For every $X\subseteq I\times J$, if $X\in\W$ then $X_i\in\pi_2(\W)$
for infinitely many $i\in I$.
\end{enumerate}
\end{theorem}

\begin{proof}
$(i)\Leftrightarrow(ii)$ directly follows from equalities 
$\pi_1(\U\otimes\V)=\U$ and $\pi_2(\U\otimes\V)=\V$.

\smallskip
$(ii)\Rightarrow(iii)$.
By the hypothesis, $\{i\in I\mid X_i\in\pi_2(\W)\}=I\in\pi_1(\W)$,
and hence $X\in\pi_1(\W)\otimes\pi_2(\W)$.

\smallskip
$(iii)\Rightarrow(ii)$.
It is enough to show the inclusion $\pi_1(\W)\otimes\pi_2(\W)\subseteq\W$.
By definition, $X\in\pi_1(\W)\otimes\pi_2(\W)$ if and only if
$A:=\{i\in I\mid X_i\in\pi_2(\W)\}\in\pi_1(\W)$.
Now let $Y:=X\cup(A^c\times J)$. Note that $Y_i\in\pi_2(\W)$
for all $i\in I$ and so, by our assumption, $Y\in\W$.
Finally, observe that $A\in\pi_1(\W)\Leftrightarrow\pi_1^{-1}(A)=A\times J\in\W$,
and we can conclude that $X=Y\cap(A\times J)\in\W$.

\smallskip
$(iii)\Leftrightarrow(iv)$. By considering the contrapositive implication, 
one can reformulate $(iii)$ as follows:
``For every $X\subseteq I\times J$, if $X\notin\W$
then there exists $i\in I$ with $X_i\notin\pi_2(\W)$."
Passing to complements, by the property of ultrafilter
the previous statement is equivalent to:
``For every $X\subseteq I\times J$, if $X^c\in\W$ then
there exists $i\in I$ with $(X_i)^c=(X^c)_i\in\pi_2(\W)$". 
It is then readily verified that this last statement is in turn equivalent to $(iv)$.

\smallskip
$(v)\Leftrightarrow(vi)$ is proved similarly to the equivalence $(iii)\Leftrightarrow(iv)$.

\smallskip
$(v)\Rightarrow(iii)$ and $(vi)\Rightarrow(iv)$ are trivial.

\smallskip
$(ii)\Rightarrow(v)$.
Let $X\subseteq I\times J$, and assume that  
$A:=\{i\in I\mid X_i\in\pi_2(\W)\}$ is cofinite. Since $\pi_1(\W)$ is non-principal, 
$A\in\pi_1(\W)$, and this means that $X\in\pi_1(\W)\otimes\pi_2(\W)$.
By our assumption, we conclude that $X\in\W$.
\end{proof}

In the special case where the sets of indexes $I=J=\N$,
another relevant equivalent condition for an ultrafilter to be a tensor product
was proved in 1971 by C. Puritz \cite[Theorem\,3.4]{pu} in the language of nonstandard analysis.
By using the same ideas of Puritz', reformulated in the language of ultrafilters,
one can extend his result to arbitrary sets $J$, provided that the ultrafilter 
$\pi_2(\W)$ on $J$ is countably incomplete.

\smallskip
Recall that an ultrafilter $\U$ on $I$ is \emph{countably incomplete}
if it is not closed under countable intersections, \emph{i.e.}
if there exists a countable family $\{A_n\mid n\in\N\}\subset\U$ such
that $\bigcap_{n\in\N}A_n\notin\U$.
It is easy to show that an ultrafilter $\U$ on $I$ is countably incomplete
if and only if one can obtain $I=\bigcup_{n\in\N}B_n$ as a countable
union of sets $B_n\notin\U$.\footnote
{~If $\{A_n\mid n\in\N\}\subseteq\U$ is a family whose intersection
$X:=\bigcap_{n\in\N}A_n\notin\U$, then the set $I=X\cup\bigcup_{n\in\N}(I\setminus A_n)$
is a countable union of sets that do not belong to $\U$.
Conversely, given $I=\bigcup_{n\in\N}B_n$ where every $B_n\notin\U$, 
the countable intersection $\bigcap_{n\in\N}(I\setminus B_n)=\emptyset\notin\U$,
while every set $I\setminus B_n\in\U$.}

\begin{remark}\label{remark-countablyincomplete}
Trivially, every principal ultrafilter is countably complete. 
We observe that
every non-principal ultrafilter $\U$ on $\N$ is countably incomplete,
since $\N=\bigcup_{n\in\N}\{n\}$ where every singleton $\{n\}\notin\U$.
We remark that there may exist non-principal ultrafilters on $I$ that are countably complete
only when the cardinality of $I$ is really ``large", in fact, 
at least as large as the first measurable cardinal.\footnote
{~Recall that a cardinal $\kappa$ is \emph{measurable} if there exists a non-principal
ultrafilter on $\kappa$ which is closed under intersections of size smaller than $\kappa$.
One can prove that if $\kappa>\aleph_0$ is measurable then $\kappa$ is 
weakly compact, and is the $\kappa$-th inaccessible cardinal. 
In general, for large cardinals, the notion of countable incompleteness is stronger than non-principality.
(For a treatment of this set-theoretic topic, see \emph{e.g.} \cite[\S 4.2]{ck} and \cite[Ch.10]{je}).} 
In particular, it is consistent with \textsf{ZFC}
that all non-principal ultrafilters are countably incomplete.
In any case, if we work with sets that are not too ``large" 
(\emph{i.e.} of size less than the first measurable cardinal, if any)
then the properties of non-principality and of countable incompleteness are equivalent.
\end{remark}


\begin{theorem}
Let $\W$ be an ultrafilter on $\N\times J$, and assume that
$\pi_2(\W)$ is countably incomplete. Then the following properties
are equivalent:
\begin{enumerate}
\item
$\W$ is a tensor product.
\item
For every $f:J\to \N$, if the ultrafilter $f(\pi_2(\W))$ on $\N$ is non-principal then
$\Gamma(f):=\{(n,j)\in\N\times J\mid n<f(j)\}\in\W$.
\end{enumerate}
\end{theorem}

\begin{proof}
$(1)\Rightarrow(2)$. 
This implication does not need any assumption on $\pi_2(\W)$.
Let $\W$ be a tensor product. By $(ii)$ in the previous theorem
it is $\W=\pi_1(\W)\otimes\pi_2(\W)$.
Now let $f:J\to\N$, and assume that $f(\pi_2(\W))$ is a non-principal ultrafilter on $\N$.
Then for all $n\in\N$ we have
$\{k\in\N\mid n<k\}\in f(\pi_2(\W))$. This means that
$f^{-1}(\{k\in\N\mid n<k\})=\Gamma(f)_n\in\pi_2(\W)$.
Since $\{n\in\N\mid\Gamma(f)_n\in\pi_2(\W)\}=\N\in\pi_1(\W)$, we can conclude that
$\Gamma(f)\in\pi_1(\W)\otimes\pi_2(\W)=\W$.

\smallskip
$(2)\Rightarrow(1)$.
We show that $\W$ is a tensor product by using $(iii)$ of 
Theorem \ref{tensorproductsequivalences}. So, let us assume that
$X\subseteq\N\times J$ is such that $X_n\in\pi_2(\W)$
for all $n\in\N$. We have to show that $X\in\W$.
Since $\pi_2(\W)$ is a countably incomplete ultrafilter on $J$,
we can pick a function $\varphi:J\to\N$ such that $\varphi(\pi_2(\W))$ is non-principal.
Then define $f:J\to\N$ by setting:
$$f(j):=\min\{n\in\N\mid n\ge\varphi(j)\ \text{or}\ (n,j)\notin X\}.$$
We observe that also the ultrafilter $f(\pi_2(\W))$ is non-principal.
Indeed, for every $n\in\N$, we have 
$\Lambda_n:=\{j\in J\mid n<\varphi(j)\}=\varphi^{-1}((n,+\infty))\in\pi_2(\W)$,
since the cofinite set $(n,+\infty)\in\varphi(\pi_2(\W))$.
By our assumptions $X_n\in\pi_2(\W)$,
and so also $X_n\cap\Lambda_n=\{j\in J\mid n<\varphi(j)\ \text{and}\ (n,j)\in X\}\in\pi_2(\W)$.
Now notice that
$X_n\cap\Lambda_n\subseteq\{j\in J\mid n<f(j)\}=f^{-1}((n,+\infty))\in\pi_2(\W)$,
and hence $(n,+\infty)\in f(\pi_2(\W))$. This completes the proof
that $f(\pi_2(\W))$ is non-principal and so, by the hypothesis,
$\Gamma(f)\in\W$. Finally, note that
$\Gamma(f)\subseteq\{(n,j)\in\N\times J\mid n<\varphi(j)\ \text{and}\ (n,j)\in X\}\subseteq X$,
and therefore we can conclude that $X\in\W$, as desired.
\end{proof}

\begin{remark}
We observe that the arguments used in the previous proof do not work
if one replaces $\N$ with an infinite ordered set $(I,<)$ not isomorphic to $(\N,<)$.
Indeed, notice first that the very definition of $f$ assumes the existence of a least element,
and assuming this for arbitrary subsets is the same as asking for $(I,<)$ to be well-ordered.
Besides, in the last part of the proof the following property is assumed: 
``For every $n\in\N$ and for every non-principal ultrafilter $\V$ on $\N$,
the end-segment $(n,+\infty)\in\V$." If the infinite well-ordered set $(I,<)$ is
not isomorphic to $(\N,<)$ then there exists $\xi\in I$ such that
its initial segment $S_\xi=\{x\in I\mid x<\xi\}$ is infinite.
But then there exist non-principal ultrafilters $\V$ on $I$ such that $S_\xi\in\V$,
and hence the complement $I\setminus S_\xi=(\xi,+\infty)\notin\V$.
\end{remark}


\begin{notation}
We denote by $\pi_{1,2},\pi_{2,3},\pi_{1,3}$
the projections defined on a triple Cartesian product $I_1\times\ I_2\times I_3$
where $\pi_{1,2}(i_1,i_2,i_3)=(i_1,i_2)$, $\pi_{2,3}(i_1,i_2,i_3)=(i_2,i_3)$,
and $\pi_{1,3}(i_1,i_2,i_3)=(i_1,i_3)$.
\end{notation}

\begin{corollary}
Let $\W,\W'$ be ultrafilters on $\N\times J$ where $\pi_1(\W),\pi_1(\W')$ are countably incomplete.
If $\W$ is a tensor product and there is an ultrafilter $\Zc$ on $\N\times\N\times I$
such that:
\begin{itemize}
\item
$\pi_{1,3}(\Zc)=\W'$, 
\item
$\pi_{2,3}(\Zc)=\W$,
\item
$\{(n',n,j)\in\N\times\N\times J\mid n'<n\}\in\Zc$,
\end{itemize}
then also $\W'$ is a tensor product.
\end{corollary}

\begin{proof}
Since $\W$ is a tensor product,
for every $f:J\to\N$ the set
$$\Gamma(f)=\{(n,j)\mid n<f(j)\}\in\W.$$ 
By the hypotheses,
it follows that $\Lambda:=\{(n',n,j)\mid n'<n\ \text{and}\ n<f(j)\}\in\Zc$.
But then also the superset $\Lambda':=\{(n',n,j)\mid n'<f(j)\}\in\Zc$,
and hence $\pi_{1,2}(\Lambda')=\Gamma(f)\in\W'$.
\end{proof}

The following proposition lists some general properties of tensor products.

\begin{proposition}\label{tensorproperties}
\


\begin{enumerate}
\item
If the ultrafilter $\W$ on $I\times J$ is such that $\pi_1(\W)$
or $\pi_2(\W)$ is principal, then $\W$ is a tensor product.
\item
If $f_\ell:I_\ell\to J_\ell$ and $\U_\ell$ is an ultrafilter on $I_\ell$ for $\ell=1,\ldots,k$ then 
$$f_1(\U_1)\otimes\ldots\otimes f_k(\U_k)=
(f_1,\ldots,f_k)(\U_1\otimes\ldots\otimes\U_k)$$ 
where $(f_1,\ldots,f_k):I_1\times\ldots\times I_k\to J_1\times\ldots J_k$ is
the function $(i_1,\ldots,i_k)\mapsto(f_1(i_1),\ldots,f_k(i_k))$.
\item
Let $\Zc$ be an ultrafilter on $I\times I'\times J$.
If $\pi_{1,3}(\Zc)$ and $\pi_{2,3}(\Zc)$ are tensor products on $I\times J$
and $I'\times J$ respectively, then $\Zc=\pi_1(\Zc)\otimes\pi_2(\Zc)\otimes\pi_3(\Zc)$ 
is a tensor product on $I\times I'\times J$.
\end{enumerate}
\end{proposition}

\begin{proof}
(1). Assume first that $\pi_1(\W)=\U_{i_0}$ is the principal ultrafilter on $I$
generated by an element $i_0\in I$. Then $\{i_0\}\times J=\pi_1^{-1}(\{i_0\})\in\W$.
Now, for every $X\in\W$ one has that $X\cap(\{i_0\}\times J)=\{i_0\}\times X_{i_0}\in\W$,
and hence $X_{i_0}\in\pi_2(\W)$.
By $(iv)$ of Theorem \ref{tensorproductsequivalences}
this shows that $\W$ is a tensor product.

When the second projection $\pi_2(\W)=\U_{j_0}$ is the principal
ultrafilter on $J$ generated by an element $j_0$ the proof is similar.
Indeed, in this case $I\times\{j_0\}=\pi_2^{-1}(\{j_0\})\in\W$. Now let us assume that
$X_i\in\pi_2(\W)=\U_{j_0}$ for every $i\in I$.
Then $j_0\in X_i$, \emph{i.e.} $(i,j_0)\in X$, for every $i\in I$.
So, $I\times\{j_0\}\subseteq X$ and we can conclude that $X\in\W$.
By $(iii)$ of Theorem \ref{tensorproductsequivalences}
this shows that $\W$ is a tensor product.

\smallskip
(2). We prove here only the case $k=2$; the general case 
will then easily follow by associativity.
We use the characterization of tensor pairs as given 
by $(iv)$ of Theorem \ref{tensorproductsequivalences}.
Let $Y\in(f_1,f_2)(\U_1\otimes\U_2)$. 
We want to show that there exists $y\in J_1$ such that 
$Y_y\in\pi_2((f_1,f_2)(\U_1\otimes\U_2))$, \emph{i.e.} $Y_y\in f_2(\U_2)$, since
$\pi_2\circ(f_1,f_2)=f_2\circ\pi_2$ and $f_2(\pi_2(\U_1\otimes\U_2))=f_2(\U_2)$.
The set $X:=(f_1,f_2)^{-1}(Y)\in\U_1\otimes\U_2$, and so there exists $x\in I_1$
such that 
$$X_x=\{i\in I_2\mid (f_1(x),f_2(i))\in Y\}=\{i\in I_2\mid f_2(i)\in Y_{f_1(x)}\}=f_2^{-1}(Y_{f_1(x)})\in\U_2.$$
If we let $y=f_1(x)\in J_1$ then we have shown that $Y_y\in f_2(\U_2)$, as desired.

\smallskip
(3). Also in this case we use the characterization of tensor pairs as given 
by $(iv)$ of Theorem \ref{tensorproductsequivalences}.
Let $Y\in\Zc$. We have to show that there exists $(x,x')\in I\times I'$
such that $Y_{(x,x')}=\{j\in J\mid (x,x',j)\in Y\}\in\pi_3(\Zc)$,
where $\pi_3:(i,i',j)\mapsto j$ is the projection on the third coordinate
of $I\times I'\times J$. 
Denote by $\tau:(i,j)\mapsto j$ and by $\vartheta:(i',j)\mapsto j$ the projections of
$I\times J$ and of $I'\times J$, respectively, on the second coordinate.
Note that $\tau\circ \pi_{1,3}=\vartheta\circ \pi_{2,3}=\pi_3$.
Now $\pi_{1,3}(Y)\in\pi_{1,3}(\Zc)$, which is a tensor product on $I\times J$,
and so there exists $x\in I$ such that $[\pi_{1,3}(Y)]_x\in\tau(\pi_{1,3}(\Zc))=\pi_3(\Zc)$.
Similarly, $\pi_{2,3}(Y)\in\pi_{2,3}(\Zc)$, which is a tensor product on $I'\times J$,
and so there exists $x'\in I'$ such that $[\pi_{2,3}(X)]_{x'}\in\vartheta(\pi_{2,3}(\Zc))=\pi_3(\Zc)$.
Then the intersection $[\pi_{1,3}(Y)]_x\cap[\pi_{2,3}(Y)]_{x'}\in\pi_3(\Zc)$.
Finally, observe that $[\pi_{1,3}(Y)]_x\cap[\pi_{2,3}(Y)]_{x'}=\{j\in J\mid (x,x',j)\in Y\}=Y_{(x,x')}$.
\end{proof}

\medskip
\section{Tensor products and Ramsey's Theorem}

In this section we will see that the property isolated in  Ramsey's Theorem, 
a cornerstone of combinatorics, is closely related to tensor products of ultrafilters.

\begin{notation}
For every set $I$ and for every $k\in\N$, denote by 
$$[I]^k=\{A\subseteq X\mid |A|=k\}$$
the set of all $k$-subsets of $I$. When $k=1$ we identify $[I]^1=I$.
\end{notation}

\smallskip
\noindent
\textbf{Theorem} (Ramsey).
\emph{Let $I$ be an infinite set. For every $k$ and for every finite coloring $[I]^k=C_1\cup\ldots\cup C_r$
there exists an infinite $H\subseteq I$ and a color $C_i$ such that $[H]^k\subseteq C_i$.}

\smallskip
An infinite set $H$ as above is called \emph{homogeneous} for the considered coloring.

We will see that tensor powers of non-principal ultrafilters
are ``witnesses" of the partition regularity property given by Ramsey's Theorem.

\begin{notation}
For every ultrafilter $\U$, its $k$-iterated tensor power is denoted by
$$\U^{k\otimes}=\underbrace{\U\otimes\cdots\otimes\U}_{k\ \text{times}}.$$
\end{notation}

\begin{proposition}\label{superdiagonal}
For every non-principal ultrafilter $\U$ on $\N$ and for every $k\in\N$, the
superdiagonal $\Delta^+_k=\{(n_1,\ldots,n_k)\mid n_1<\ldots<n_k\}\in\U^{k\otimes}$
belongs to the $k$-th tensor power of $\U$.
\end{proposition}

\begin{proof}
We proceed by induction on $k$ and prove that for every $\ell\in\N$
the set 
$$\{(n_1,\ldots,n_k)\mid \ell<n_1<\ldots<n_k\}\in\U^{k\otimes}.$$
Then clearly also the superset $\Delta^+_k\in\U^{k\otimes}$.
The base case $k=1$ is trivial because $\{n_1\in\N\mid \ell<n_1\}$
is cofinite, and hence it belongs to every non-principal ultrafilter $\U=\U^{1\otimes}$.
At the inductive step, by definition
$\Delta^+_{k+1}\in\U^{(k+1)\otimes}=\U\otimes\U^{k\otimes}$ if and only if the set
$\Gamma:=\{n\in\N\mid (\Delta^+_{k+1})_n\in\U^{k\otimes}\}\in\U$, and
this is true because $\Gamma=\N$.
Indeed, for every $\ell\in\N$, by the inductive hypothesis
we know that the following vertical fiber belongs to $\U^{k\otimes}$:
\begin{multline*}
(\Delta^+_{k+1})_\ell=\{(n_1,\ldots,n_k)\in\N^k\mid (\ell,n_1,\ldots,n_k)\in\Delta^+_{k+1}\}=
\\
=\{(n_1,\ldots,n_k)\in\N^k\mid \ell<n_1<\ldots<n_k\}\in\U^{k\otimes}.
\end{multline*}
\end{proof}

We observe that if the property of Ramsey's Theorem holds for a set $I$ then
it also holds for all supersets $I'\supseteq I$. In consequence,
without loss of generality, one can focus on the ``smallest" infinite set of indexes,
namely $I=\N$.

\begin{theorem}\label{UotimesU}
Let $\U$ be a non-principal ultrafilter on $\N$.
For every $X\in\U^{k\otimes}$ there exists an
infinite set $H=\{h_1<h_2<\ldots<h_n<h_{n+1}<\ldots\}$ such that
$$[H]^k:=\{(h_{s_1},\ldots,h_{s_k})\mid s_1<\ldots<s_k\}\subseteq X.$$

Conversely, if $X\subseteq\N^k$ is such that there exists 
an infinite set $H$ with $[H]^k\subseteq X$
then there exists a non-principal ultrafilter $\U$ on $\N$ with $X\in\U^{k\otimes}$.
\end{theorem}

\begin{proof}
This is a particular case of Theorem \ref{main} that will be stated in Section 6.
However, as a warm-up towards the general result, we give here a direct proof 
of the above property in the case $k=2$. 

Let $\U$ be a non-principal ultrafilter on $\N$, and
let $X\in\U\otimes\U$. By induction on $n$, 
we define elements $h_n\in\N$ in such a way that:
\begin{itemize}
\item
$h_n\ne h_i$ for all $i<n$,
\item
The vertical fiber $X_{h_n}\in\U$,
\item
$(h_i,h_j)\in X$ for all $i<j\le n$.
\end{itemize}

By definition, $X\in\U\otimes\U$ if and only if $X_\U:=\{i\in I\mid X_i\in\U\}\in\U$.
At the base step pick an element $h_1\in X_\U$, so that $X_{h_1}\in\U$. 
At the inductive step $n+1$, we observe that 
$\Gamma_n:=X_\U\cap\bigcap_{i=1}^n X_{h_i}\in\U$.
Since $\U$ is non-principal, the set $\Gamma_n$ is infinite and we can pick
$h_{n+1}\in\Gamma_n$ such that $h_{n+1}>h_n$.
It is readily verified that the set $H=\{h_n\mid n\in\N\}$
satisfies the desired properties.

Conversely, let $X\subseteq\N\times\N$ and let
$H=\{h_1<h_2<\ldots<h_n<h_{n+1}<\ldots\}$ be an infinite set
such that $[H]^2\subseteq X$.
The family $\{\{h_n\mid n\ge k\}\mid k\in\N\}$ has the finite intersection property,
and so we can pick an ultrafilter $\U$ that extends it. Note that such a $\U$ is 
non-principal. Finally, it is directly verified from the definitions that
$[H]^2=\{(h_n,h_m)\mid n<m\}\in\U\otimes\U$, and hence also the superset $X\in\U\otimes\U$.
\end{proof}

As a straight application of Theorem \ref{UotimesU}
we now obtain an ultrafilter proof of Ramsey's Theorem
(see \emph{e.g.} \cite[Theorem 3.3.]{ck}).

In the following we will identify each unordered $k$-tuple $\{n_1<\ldots<n_k\}\in[\N]^k$
with the corresponding ordered $k$-tuple $(n_1,\ldots,n_k)\in\N^k$ where components
are arranged in increasing order, so that
$$[\N]^k=\Delta^+_k:=\{(n_1,\ldots,n_k)\in\N^k\mid n_1<\ldots<n_k\}.$$

%

\begin{theorem}[Ramsey]
Let $I$ be an infinite set and let $k\in\N$. For every finite coloring
$[I]^k=C_1\cup\ldots\cup C_r$ there exists an infinite set $H$ 
and a color $C_i$ such that $[H]^k\subseteq C_i$.
\end{theorem}

\begin{proof}
As already observed, without loss of generality we can assume $I=\N$.
Pick a non-principal ultrafilter $\U$ on $\N$.
By Proposition \ref{superdiagonal}, the set $[\N]^k=\Delta^+_k\in\U^{k\otimes}$
and so, by the property of ultrafilter, there exists a color $C_i\in\U^{k\otimes}$. 
Then apply the first part of Theorem \ref{UotimesU} to find an infinite set $H$
such that $[H]^2\subseteq C_i$.
\end{proof}

The consequences of Ramsey's Theorem are far-reaching, and there is
an extensive bibliography on its applications in diverse areas of mathematics. 
Here we give just one relevant example in the context of calculus,
that we did not find explicitly mentioned in the literature.

%


\begin{proposition}[Weierstrass]
Every bounded sequence $(a_n)$ of real numbers
admits a Cauchy subsequence $(a_{n_k})$.
\end{proposition}

\begin{proof}
Consider the partition $[\N]^3=C_1\cup C_2$ where
$$C_1=\left\{\{k<n<m\}\,\Big|\, |a_n-a_m|\le\frac{1}{k}\right\}\ \text{and}\ 
C_2=\left\{\{k<n<m\}\,\Big|\, |a_n-a_m|>\frac{1}{k}\right\}.$$
By Ramsey's Theorem there exists an increasing sequence $(n_s)$
that is homogenous for the partition.
We observe that $[(n_s)]^3\subseteq C_2$ is not possible,
as otherwise $(a_{n_s})$ would be unbounded.
To see this, assume for the sake of contradiction that
there exists an interval $I$ of finite length $M\in\N$ such that
$(a_{n_s})\subseteq I$, and consider the set $X=\{a_{n_s}\mid k=2,\ldots,Mn_1+2\}$.
Since $X$ contains $Mn_1+1$ points and is included in an interval of length $M$,
there must be at least two distinct points $x,y\in X$ such that 
$|x-y|\le\frac{M}{Mn_1}=\frac{1}{n_1}$.
On the other hand, we also have 
$|x-y|=|a_{n_k}-a_{n_h}|>\frac{1}{n_1}\ge|x-y|$, a contradiction.

So, it must be $[(n_k)]^2\subseteq C_1$, and therefore the subsequence $(a_{n_k})$ is Cauchy.
Indeed, given $\varepsilon>0$ pick $k$ such that $\frac{1}{n_k}<\varepsilon$.
Then for all $s>t>k$ we have 
$|a_{n_s}-a_{n_t}|\le\frac{1}{n_k}<\varepsilon$.
\end{proof}

The property of Theorem \ref{UotimesU} can be reformulated in topological terms.

Recall that the topological space of ultrafilters $\beta\N^k:=\{\W\mid \W\ \text{ultrafilter on}\ \N^k\}$
has a base of (cl)open sets given by the family of sets $\mathcal{O}_X:=\{\W\in\beta\N^k\mid X\in\W\}$
for $X\subseteq\N^k$ (see, \emph{e.g.}, \cite[Chapter 3]{hs}).

\begin{theorem}
Let $\W\in\beta\N^k$. The following properties are equivalent:
\begin{enumerate}
\item
$\W$ is a ``Ramsey's witness", \emph{i.e.} for every $X\in\W$ there exists an infinite $H$ with $[H]^k\subseteq X$.
\item
$\W$ belongs to the topological closure of the space of non-trivial tensor powers
$\{\U^{k\otimes}\mid \U\in\beta\N\setminus\N\}$.
\end{enumerate}
\end{theorem}

\begin{proof}
Call a set $X\subseteq\N^k$ \emph{Ramsey-large} if there exists an infinite $H$
with $[H]^k\subseteq X$. According to Theorem \ref{UotimesU}, a set $X\subseteq\N^k$ is Ramsey-large
if and only if there exists a non-principal ultrafilter $\U$ on $\N$ with $X\in\U^{k\otimes}$.
Then $\W$ is a Ramsey's witness if and only if every $X\in\W$ is
Ramsey-large if and only if for every basic open neighborhood $\mathcal{O}_X$ of $\W$ there exists
a non-principal ultrafilter $\U$ with $X\in\U^{k\otimes}$; and this last property
precisely means that $\W$ belongs to the topological closure of 
$\{\U^{k\otimes}\mid\U\in\beta \N^k\setminus\N^k\}$.
\end{proof}

\emph{Ramsey's witnesses} on $\N\times\N$ and 
their applications in arithmetic Ramsey theory,
are studied in detail in \cite{dlmmr}.

\smallskip
Theorem \ref{UotimesU} has the following straight generalisation,
also a particular case of the general Theorem \ref{main},
which we will see later.

\begin{theorem}\label{UotimesV}
Let $\U_1,\ldots,\U_k$ be non-principal ultrafilters on $\N$.
For every $X\in\U_1\otimes\ldots\otimes\U_k$ there exist
infinite sets $H_i=\{h_{i,1}<h_{i,2}<\ldots<h_{i,n}<\ldots\}$ for $i=1,\ldots,k$
such that
$$[H_1,\ldots,H_k]^k:=\{(h_{1,s_1},\ldots,h_{k,s_k})\mid s_1< \ldots< s_k\}\subseteq X.$$

Conversely, if $X\subseteq\N^k$ is such that there exist
infinite sets $H_i$ for $i=1,\ldots,k$ with
$[H_1,\ldots,H_k]^k\subseteq X$, then there exist non-principal 
ultrafilters $\U_1,\ldots,\U_k$ on $\N$ such that $X\in\U_1\otimes\ldots\otimes\U_k$.
\end{theorem}

In the same way as done with Ramsey's witnesses, one can reformulate in topological terms.


\begin{theorem}
Let $\W\in\beta\N^k$. The following properties are equivalent:
\begin{enumerate}
\item
For every $X\in\W$ there exists infinite sets $H_1,\ldots,H_k$ such that
\\
$[H_1,\ldots,H_k]^k\subseteq X$.
\item
$\W$ belongs to the topological closure of 
the space of non-trivial tensor products
$\{\U_1\otimes\ldots\otimes\U_k\mid \U_1,\ldots,\U_k\in\beta\N\setminus\N\}$.
\end{enumerate}
\end{theorem}

If, instead of considering only tensor powers of a non-principal
ultrafilter $\U$ on $\N$, one also admits tensor products with ultrafilters 
$\W$ on $\N\times\N$ that extend the product family
$\U\times\U:=\{A\times B\mid A,B\in\U\}$, then the corresponding sets 
still have a rich combinatorial structure. 
One example is given by following result that was used in
\cite{dr} to prove ``abundance" of monochromatic exponential triples $\{a, b, b^a\}$
in every finite coloring of $\N$.

\begin{proposition}
Let $\W$ be an ultrafilter on $\N\times\N$ such that
$\pi_1(\W)=\pi_2(\W)=\U$ (equivalently, $\W\supset\U\times\U$). 
Then for every $X\in\U\otimes\W\otimes\U$
there exists an increasing sequence $(a_s)$
such that $(a_i,a_{2j},a_{2j+1},a_k)\in X$
for all $i,j,k$ with $i<2j$ and $2j+1<k$.
\end{proposition}

\begin{proof}
For $X\subseteq\N\times\N^2\times\N$, for $n\in\N$,
and for $(m,k)\in\N\times\N$, denote:
\begin{itemize}
\item
$X_n:=\{(m,k,h)\in\N^3\mid (n,m,k,h)\in X\}$.
\item
$X_{n,m,k}:=\{h\in\N\mid (n,m,k,h)\in X\}$.
\end{itemize}
Then by the definitions we have that:
\begin{enumerate}
\item
$X\in\U\otimes\W\otimes\U\Longleftrightarrow
X':=\{n\in\N\mid X_n\in\W\otimes\U\}\in\U$.
\item
$X_n\in \W\otimes\U \Longleftrightarrow 
(X_n)_\U:=\{(m,k)\in\N^2\mid X_{n,m,k}\in\U\}\in\W$. 
\item
$n\in X'\Longleftrightarrow (X_n)_\U\in\W$.
\end{enumerate}

Now let $X\in\U\otimes\W\otimes\U$ be given. Pick $a_1\in X'\in\U$;
then $(X_{a_1})_\U\in\W$. Pick
$(a_2,a_3)\in (X_{a_1})_\U\cap(X')^2\in\W$; then
$X_{a_1,a_2,a_3}\in\U$ and 
$(X_{a_2})_\U,(X_{a_3})_\U\in\W$. Pick 
$$(a_4,a_5)\in (X')^2\cap(X_{a_1})_\U\cap(X_{a_2})_\U\cap(X_{a_3})_\U
\cap(X_{a_1,a_2,a_3})^2\in\W;$$
then $(X_{a_4})_\U,(X_{a_5})_\U\in\W$,
and $X_{a_1,a_4,a_5},X_{a_2,a_4,a_5},X_{a_3,a_4,a_5}\in\U$, and 
$$(a_1,a_2,a_3,a_4),(a_1,a_2,a_3,a_5)\in X.$$
Inductively iterate the process. Precisely, assume that
elements $a_i$ have been defined for $i<2n$ which satisfy:
\begin{itemize}
\item
$(X_{a_i})_\U\in\W$ for every $i<2n$;
\item
$X_{a_i,a_{2j},a_{2j+1}}\in\U$ for all $i<2j<2n$;
\item
$(a_i,a_{2j},a_{2j+1},a_k)\in X$ whenever $i<2j$ and $2j+1<k<2n$.
\end{itemize}
Then pick
$$(a_{2n},a_{2n+1})\in(X')^2\cap\bigcap_{i=1}^{2n-1}(X_{a_i})_\U\,\,\cap
\bigcap_{i<2j<2n}(X_{a_i,a_{2j},a_{2j+1}})^2\in\W.$$
It is then verified in a straightforward manner that the above
properties are satisfied by replacing $n$ with $n+1$, namely:
\begin{itemize}
\item
$(X_{a_i})_\U\in\W$ for every $i<2n+2$;
\item
$X_{a_i,a_{2j},a_{2j+1}}\in\U$ for all $i<2j<2n+2$;
\item
$(a_i,a_{2j},a_{2j+1},a_k)\in X$ whenever $i<2j$ and $2j+1<k<2n+2$.
\end{itemize}
\end{proof}

\medskip
\section{Tensor products and limits of double sequences}

The following notion of limit of a sequence along an ultrafilter
was first isolated in the context of topology by H. Cartan \cite{ca1,ca2}
as a tool to study convergence in the full generality of topology.
It can be used to give simple characterisations of the topological closure of a set,
of the property of compactness, and of other topological properties.

\begin{definition}
Let $X$ be a topological space and let $\U$ be an ultrafilter on $I$. A point $x\in X$ is the
\emph{$\U$-limit} (or the \emph{limit along the ultrafilter $\U$})
of the sequence $(a_i\mid i\in I)$ in $X$
if for every neighborhood $U$ of $x$ it is $\{i\in I\mid a_i\in U\}\in\U$.
In this case, one writes $\U-\lim a_i=x$,
\end{definition}

\begin{fact}
Let $X$ be a topological space. Then:
\begin{itemize}
\item
$X$ is Hausdorff if and only if
for every ultrafilter $\U$ on a set $I$ 
every sequence $(x_i\mid i\in I)$ has at most one $\U$-limit.
\item
$X$ compact if and only if
for every ultrafilter $\U$ on a set $I$ 
all sequences $(x_i\mid i\in I)$ has at least one $\U$-limit.
\item
A point $x\in\overline{Y}$ belongs to the closure of a subspace $Y\subseteq X$
if and only if there exists an ultrafilter $\U$ on a set $I$ 
and a sequence $(y_i\mid i\in I)$ in $Y$ such that $\U-\lim_i y_i=x$.
\end{itemize}
\end{fact}

For simplicity, in the rest of this section we will focus on sequences of real numbers.

Note that for sequences of real numbers $(a_n\mid n\in\N)$, one can extend 
in a natural way the definition of limit along an ultrafilter $\U$ so
to also include $\U-\lim_n a_n=+\infty$ and $\U-\lim_n a_n=+\infty$.\footnote
{~Precisely, we say that $\lim_n a_n=+\infty$ if for every $k\in\N$
the set $\{n\in\N\mid a_n>k\}\in\U$; and similarly for $\lim_n a_n=-\infty$.}

\begin{proposition}\label{limitpointscharacterization}
Let $(a_n)$ be a real sequence and let $\ell\in\R\cup\{\pm\infty\}$. 
The following are equivalent:
\begin{enumerate}
\item
$\ell$ is a limit point of $(a_n)$.
\item
There exists a non-principal ultrafilter $\U$ on $\N$ such that
$\U-\lim_n a_n=\ell$.
\end{enumerate}
Besides, also the following two properties are equivalent to each other:
\begin{enumerate}
\setcounter{enumi}{2}
\item
$\lim_{n\to\infty}a_n=\ell$.
\item
$\U-\lim_{n}a_n=\ell$ for every non-principal ultrafilter $\U$ on $\N$.
\end{enumerate}
\end{proposition}

\begin{proof}
$(1)\Rightarrow(2)$. By the hypothesis, there exists
a subsequence $(a_{n_k})$ such that $\lim_{k\to\infty}a_{n_k}=\ell$.
For every $N\in\N$ let $\Lambda(N):=\{k\in\N\mid |a_{n_k}-\ell|<1/N\}$
if $\ell\in\R$; let $\Lambda(N):=\{s\in\N\mid a_{n_k}>N\}$ if $\ell=+\infty$;
and let $\Lambda(N):=\{k\in\N\mid a_{n_k}<-N\}$ if $\ell=-\infty$.
Since the family $\{\Lambda(N)\mid N\in\N\}$ has the finite intersection property,
we can pick a non-principal ultrafilter $\U$ that extends it.
Then it is readily verified that $\U-\lim_n a_n=\ell$.

\smallskip
$(2)\Rightarrow(1)$. For every $k\in\N$, let $\Gamma(k):=\{n\in\N\mid |a_n-\ell|<1/k\}$
if $\ell\in\R$; let $\Gamma(k):=\{n\in\N\mid a_n>k\}$ if $\ell=+\infty$;
and let $\Gamma(k):=\{n\in\N\mid a_n<-k\}$ if $\ell=-\infty$.
By the hypothesis, $\Gamma(k)\in\U$ for every $k$,
and since $\U$ is non-principal, we can pick elements $n_k\in\Gamma(k)$
in such a way that $n_{k+1}>n_k$ for every $k$.
Then it is readily verified that $\lim_{k\to\infty}a_{n_k}=\ell$.

\smallskip
$(3)\Leftrightarrow(4)$. It directly follows from the equivalence $(1)\Leftrightarrow(2)$,
by recalling that $\lim_{n\to\infty}a_n=\ell$ if and only if $\ell$
is the unique limit point of $(a_n)$.
%
%
\end{proof}

Tensor products allow similar characterizations 
for double-index sequences.\footnote
{~Recall that, by definition, if $(a_{n,m}\mid (n,m)\in\N\times\N)$ is a double real sequence and $\ell\in\R$,
then $\lim_{(n,m)\to\infty}a_{n,m}=\ell$ if for every $\epsilon>0$ there exists $N\in\N$
such that $|a_{n,m}-\ell|<\epsilon$ for all $n,m\ge N$.}

\begin{theorem}\label{doublesequence}
Let $(a_{n,m}\mid n,m\in\N)$ be a double-index real sequence and let $\ell\in\R\cup\{\pm\infty\}$. 
If $\lim_{m\to\infty} a_{n,m}\in\R$ exists for every $n$, then
the following properties are equivalent to each other:

\begin{enumerate}
\item
$\lim_{n\to\infty}(\lim_{m\to\infty}a_{n,m})=\ell$.
\item
$(\U\otimes\V)-\lim_{(n,m)}a_{n,m}=\ell$ for all non-principal ultrafilters $\U$ and $\V$ on $\N$. 
\item
For every non-principal ultrafilter $\U$ on $\N$ there exists a non-principal ultrafilter $\V$ on $\N$
such that $(\U\otimes\V)-\lim_{(n,m)}a_{n,m}=\ell$.
\end{enumerate}
\end{theorem}

\begin{proof}
Denote by $\ell_n:=\lim_{m\to\infty}a_{n,m}$, 

\smallskip
$(1)\Rightarrow (2)$. 
Fix non-principal ultrafilters $\U$ and $\V$ on $\N$. For $k,n\in\N$, let 
$$\Lambda(k,n):=\begin{cases}
\{m\in\N\mid |a_{n,m}-\ell|<1/k\}
& \text{if}\ \ell\in\R;
\\
\{m\in\N\mid a_{n,m}>k\}
& \text{if}\ \ell=+\infty;
\\
\{m\in\N\mid a_{n,m}<-k\}
& \text{if}\ \ell=-\infty.
\end{cases}$$

Our goal is to show that $\Lambda(k):=\{n\in\N\mid \Lambda(k,n)\in\V\}\in\U$ for every $k\in\N$, so that
we can conclude $(\U\otimes\V)-\lim_{n,m}a_{n,m}=\ell$.

To this end, consider the sets $\Gamma(k,n):=\{m\in\N\mid |a_{n,m}-\ell_n|<1/2k\}$, and the sets
$$\Theta(k):=\begin{cases}
\{n\in\N\mid |\ell_n-\ell|<1/2k\}
& \text{if}\ \ell\in\R;
\\
\{n\in\N\mid \ell_n>2k\}
& \text{if}\ \ell=+\infty;
\\
\{n\in\N\mid \ell_n<-2k\}
& \text{if}\ \ell=-\infty.
\end{cases}$$

Now let $k\in\N$ be fixed.
Notice that the set $\Theta(k)\in\U$ is cofinite,
since $\lim_n\ell_n=\ell$.

Assume first that $\ell\in\R$. Then for every $n\in\Theta(k)$ 
and for every $m\in\Gamma(k,n)$, 
one has that 
$$|a_{n,m}-\ell|\ \le\ |a_{n,m}-\ell_n|+|\ell_n-\ell|<\frac{1}{2k}+\frac{1}{2k}=\frac{1}{k}.$$
This shows that for every $n\in\Theta(k)$, one has the inclusion $\Gamma(k,n)\subseteq\Lambda(k,n)$.
Since $\lim_m a_{n,m}=\ell_n$, the set $\Gamma(k,n)$ is cofinite,
and hence $\Lambda(k,n)\in\V$. But then $\Theta(k)\subseteq\Lambda(k)$, 
and we can finally conclude that $\Lambda(k)\in\U$, because it is a superset of a cofinite set.

In case $\ell=+\infty$, for every $n\in\Theta(k)$ 
and for every $m\in\Gamma(k,n)$ one has that
$$a_{n,m}=(a_{n,m}-\ell_n)+\ell_n>-\frac{1}{2k}+2k>k.$$
So, also in this case for every $n\in\Theta(k)$
we have the inclusion $\Gamma(k,n)\subseteq\Lambda(k,n)$,
and similarly as above, it follows that $\Lambda(k)\in\U$, as desired.
The last case when $\ell=-\infty$ is proved in the same fashion,
starting from the following inequality that holds for every $n\in\Theta(k)$ 
and for every $m\in\Gamma(k,n)$:
$$a_{n,m}=(a_{n,m}-\ell_n)+\ell_n<\frac{1}{2k}-2k<-k.$$

\smallskip
$(2)\Rightarrow(3)$ is trivial.


\smallskip
$(3)\Rightarrow (1)$. For $k\in\N$, let
$$S(k):=\begin{cases}
\{n\in\N\mid |\ell_n-\ell|\ge \frac{1}{k}\} & \text{if}\ \ell\in\R;
\\
\{n\in\N\mid \ell_n\le k\} & \text{if}\ \ell=+\infty;
\\
\{n\in\N\mid \ell_n\ge-k\} & \text{if}\ \ell=-\infty.
\end{cases}$$
Assume for the sake of contradiction that $\lim_{n\to\infty}\ell_n\ne\ell$.
Then there exists $k\in\N$ with $S(k)$ infinite. Let such $k$ be fixed.
For $n\in\N$, consider the sets
$$T(n):=\{m\in\N\mid |a_{n,m}-\ell_n|<1/2k\}.$$
Note that every $T(n)$ is cofinite, since $\lim_m a_{n,m}=\ell_n$.
Now pick a non-principal ultrafilter $\U$ on $\N$ that contains $S(k)$.
We want to show that for every non-principal ultrafilter $\V$ on $\N$,
the ultrafilter limit $(\U\otimes\V)-\lim_{n,m}a_{n,m}\ne\ell$.

Assume first that $\ell\in\R$.
We observe that for every $n\in S(k)$ and for every $m\in T(n)$ we have that
$$|a_{n,m}-\ell|\ \ge\ |\ell_n-\ell|-|\ell_n-a_{n,m}|\ge\frac{1}{k}-\frac{1}{2k}=\frac{1}{2k}.$$
For every $n\in S(k)$, the set $T(k,n)\in\V$ because it is cofinite, and so
$$\left\{n\in\N\,\Big|\,\left\{m\in\N\,\Big|\, |a_{n,m}-\ell|\ge\frac{1}{2k}\right\}\in\V\right\}\in\U,$$
since it includes $S(k)\in\U$. This shows that the set
$\{(n,m)\in\N\times\N\mid |a_{n,m}-\ell|\ge 1/2k\}\in\U\otimes\V$, and hence
$(\U\otimes\V)-\lim_{(n,m)}a_{n,m}\ne\ell$.

If $\ell=+\infty$, for every $n\in S(k)$ and for every $m\in T(k,n)$, we have that
$$a_{n,m}=(a_{n,m}-\ell_n)+\ell_n< \frac{1}{2k}+k<2k.$$
This shows that the set
$\{(n,m)\in\N\times\N\mid a_{n,m}<2k\}\in\U\otimes\V$, and hence
$(\U\otimes\V)-\lim_{(n,m)}a_{n,m}\ne +\infty$.
The final case $\ell=-\infty$ is proved in the same fashion, starting
from the following inequality that holds for every $n\in S(k)$ and for every $m\in T(k,n)$:
$$a_{n,m}=(a_{n,m}-\ell_n)+\ell_n>-\frac{1}{k}-k>-2k.$$

\end{proof}

An interesting example of ultrafilter limit along a tensor product
can be given by considering Riemann integral.

\begin{corollary}
Assume the function $f:\R\to\R$ is Riemann integrable on every bounded interval,
and let $\ell\in\R\cup\{\pm\infty\}$. Then the following are equivalent:
\begin{enumerate}
\item
$\int_{-\infty}^{+\infty}f(x) dx=\ell$.
\item
$(\U\otimes\V)-\lim_{(n,m)}\frac{1}{m}\sum_{k=-nm}^{nm}f\left(\frac{k}{m}\right)=\ell$
for all non-principal ultrafilters $\U,\V$ on $\N$.
\item
For every non-principal ultrafilter $\U$ on $\N$ there exists a non-principal ultrafilter $\V$ on $\N$
such that $(\U\otimes\V)-\lim_{(n,m)}\frac{1}{m}\sum_{k=-nm}^{nm}f\left(\frac{k}{m}\right)=\ell$.
\end{enumerate}
\end{corollary}

\begin{proof}
It directly follows from the previous theorem, by noticing that 
$$\int_{-\infty}^{+\infty}f(x)dx=\lim_{n\to\infty}\int_{-n}^n f(x)dx=
\lim_{n\to\infty}\left(\lim_{m\to\infty}\sum_{k=-nm}^{nm}f\left(\frac{k}{m}\right)\right).$$
\end{proof}

%


\medskip
\section{Tensor products and limsup, liminf, and asymptotic density}

Infimum and supremum of real sequences admit simple characterisations in terms
of ultrafilter limits.

\begin{proposition}\label{zero}
Let $(a_n)_n$ be a real sequence and let $\ell\in\R\cup\{\pm\infty\}$.
The following are equivalent:

\begin{enumerate}
\item
$\inf_{n\in\N}a_n=\ell$.
\item
$\U-\lim_n\left(\min_{1\le k\le n}a_k\right)=\ell$ for every non-principal ultrafilter $\U$ on $\N$.
\end{enumerate}

Similarly, also the following are equivalent to each other:
\begin{enumerate}
\setcounter{enumi}{2}
\item
$\sup_{n\in\N}a_n=\ell$.
\item
$\U-\lim_n\left(\max_{1\le k\le n}a_k\right)=\ell$ for every non-principal ultrafilter $\U$ on $\N$.
\end{enumerate}
\end{proposition}

\begin{proof}
Observe that:
\begin{itemize}
\item
$\inf_{n\in\N}a_n=\lim_{n\to\infty}(\min_{1\le k\le n}a_k)=\ell$,
\item
$\sup_{n\in\N}a_n=\lim_{n\to\infty}(\max_{1\le k\le n}a_k)=\ell$.
\end{itemize}
Then apply the characterization of limit as given by
Proposition \ref{limitpointscharacterization}.
\end{proof}


\begin{example}
Recall the \emph{Schnirelmann density} $\sigma(A)$ of a set $A\subseteq\N$:
$$\sigma(A)\ :=\ \inf_{n\in\N}\frac{|A\cap[1,n]|}{n}.$$
So, for every non-principal ultrafilter $\U$ on $\N$:
$$\sigma(A)=\U-\lim_n\left(\min_{1\le k\le n}\frac{|A\cap[1,k]|}{k}\right).$$
\end{example}

By using ultrafilter limits along tensor products, one
can also give a simple characterisation of limit inferiors and limit superiors.

\begin{proposition}\label{liminflimsup}
Let $(a_n)_n$ be a bounded real sequence and let $\ell\in\R$.
The following are equivalent:

\begin{enumerate}
\item
$\liminf_{n\to\infty}a_n=\ell$.
\item
$(\U\otimes\V)-\lim_{(n,m)}(\min_{n\le k\le m}a_k)=\ell$ for all
non-principal ultrafilters $\U$ and $\V$ on $\N$.
\item
For every non-principal ultrafilter $\U$ on $\N$ there exists a non-principal
ultrafilter $\V$ on $\N$ such that $(\U\otimes\V)-\lim_{(n,m)}(\min_{n\le k\le m}a_k)=\ell$.
\end{enumerate}

Similarly, also the following are equivalent to each other:
\begin{enumerate}
\setcounter{enumi}{3}
\item
$\limsup_{n\to\infty}a_n=\ell$.
\item
$(\U\otimes\V)-\lim_{(n,m)}(\max_{n\le k\le m}a_k)=\ell$ for all
non-principal ultrafilters $\U$ and $\V$ on $\N$.
\item
For every non-principal ultrafilter $\U$ on $\N$ there exists a non-principal
ultrafilter $\V$ on $\N$ such that $(\U\otimes\V)-\lim_{(n,m)}(\max_{n\le k\le m}a_k)=\ell$.
\end{enumerate}
\end{proposition}

\begin{proof}
Observe that
$$\liminf_{n\to\infty}a_n\ =\ \lim_{n\to\infty}\inf_{k\ge n}a_k\ =\ 
\lim_{n\to\infty}\left(\lim_{m\to\infty}\min_{n\le k\le n+m}a_k\right).$$
$$\limsup_{n\to\infty}a_n\ =\ \lim_{n\to\infty}\sup_{k\ge n}a_k\ =\ 
\lim_{n\to\infty}\left(\lim_{m\to\infty}\max_{n\le k\le n+m}a_k\right).$$
Then apply Theorem \ref{doublesequence} to the double-indexed sequences
$b_{n,m}=\min_{n\le k\le n+m}a_k$ and $b'_{n,m}=\max_{n\le k\le n+m}a_k$ respectively,
where we agree that $b_{n,m}=b'_{n,m}=0$ for $n>m$.\footnote
{~Note that for all non-principal ultrafilters $\U, \V$ on $\N$,
the ``under-diagonal" 
$$\{(n,m)\in\N\times\N\mid n<m\}\notin\U\otimes\V,$$
and so the values of the sequences $b_{n,m}, b'_{n,m}$ for those indexes 
do no effect the ultrafilter limit.}
\end{proof}

%

\begin{example}\label{ex-lowerupperdensity}
Recall the \emph{lower density} $\underline{d}(A)$
and the \emph{upper density} $\overline{d}(A)$
of a set $A\subseteq\N$:
$$\underline{d}(A)\ =\ \liminf_{n\to\infty}\frac{|A\cap[1,n]|}{n}\quad\text{and}\quad
\overline{d}(A)\ =\ \limsup_{n\to\infty}\frac{|A\cap[1,n]|}{n}.$$
So, for all non-principal ultrafilters $\U,\V$ on $\N$ one has:
$$\underline{d}(A)\ =\ (\U\otimes\V)-\lim_{(n,m)}\left(\min_{n\le k\le m}\frac{|A\cap[1,k]|}{k}\right);$$
$$\overline{d}(A)\ =\ (\U\otimes\V)-\lim_{(n,m)}\left(\max_{n\le k\le m}\frac{|A\cap[1,k]|}{k}\right).$$
\end{example}


%

\begin{example}\label{ex-Banachdensity}
Recall the notions of \emph{lower} and \emph{upper Banach density} of a set $A\subseteq\N$:
$$\underline{\text{BD}}(A)\ =\ \lim_{n\to\infty}\left(\min_{x\in\N}\frac{|A\cap[x+1,x+n]|}{n}\right);$$
$$\overline{\text{BD}}(A)\ =\ \lim_{n\to\infty}\left(\max_{x\in\N}\frac{|A\cap[x+1,x+n]|}{n}\right).$$
We remark that the above limits actually exist. Indeed, if we let
\begin{itemize}
\item
$\underline{a}_n=\min_{x\in\N}|A\cap[x+1,x+n]|$, and
\item
$\overline{a}_n=\max_{x\in\N}|A\cap[x+1,x+n]|$,
\end{itemize}
then it is shown that
\begin{itemize}
\item
$\underline{\text{BD}}(A)=\lim_{n\to\infty}\frac{\underline{a}_n}{n}=\sup_{n\in\N}\frac{\underline{a}_n}{n}$, and
\item
$\overline{\text{BD}}(A)=\lim_{n\to\infty}\frac{\overline{a}_n}{n}=\inf_{n\in\N}\frac{\overline{a}_n}{n}$.
\end{itemize}
\end{example}

In consequence, by using the previous results,
one directly obtains the following characterizations.

\begin{proposition}
Let $A\subseteq\N$. 

\begin{enumerate}
\item
For all nonprincipal ultrafilters $\U$ and $\V$ on $\N$, one has
$$(\U\otimes\V)-\lim_{(n,m)}
\left[\max_{1\le k\le n}\left(\min_{1\le x\le m}\frac{|A\cap[x+1,x+k]|}{k}\right)\right] =\ 
\underline{\text{BD}}(A).$$
\item
For all nonprincipal ultrafilters $\U$ and $\V$ on $\N$, on has
$$(\U\otimes\V)-\lim_{(n,m)}
\left[\min_{1\le k\le n}\left(\max_{1\le x\le m}\frac{|A\cap[x+1,x+k]|}{k}\right)\right] =\ 
\overline{\text{BD}}(A).$$
\end{enumerate}
\end{proposition}

\begin{proof}
$(1)$.
Recall that $\overline{\text{BD}}(A)=\inf_{n\in\N}\frac{\underline{a}_n}{n}=
\lim_{n\to\infty}\min_{s\le n}\frac{\underline{a}_s}{s}$
where $\overline{a}_n:=\max_{x\in\N}|A\cap[x+1,x+n]|$.
Observe that for every 
$n\in\N$, one has 
$\frac{\overline{a}_n}{n}=\sup_{m\in\N}a_{n,m}=\lim_{m\to\infty}\max_{t\le m}a_{s,t}$
where $a_{n,m}:=\frac{|A\cap[m+1,m+n]|}{n}$, and so
$$\min_{s\le n}\frac{\overline{a}_s}{s}\ =\ \min_{s\le n}\left(\lim_{m\to\infty}\max_{t\le m}a_{s,t}\right)\ =\ 
\lim_{m\to\infty}\min_{s\le n}\left(\max_{t\le m}a_{s,t}\right).$$ 
Then $\overline{\text{BD}}(A)=\lim_{n\to\infty}\lim_{m\to\infty} b_{n,m}$ 
where $b_{n,m}:=\min_{s\le n}(\max_{t\le m}a_{s,t})$. 
Note that the double real sequence $(b_{n,m}\mid (n,m)\in\N\times\N)$
is bounded and such that $\lim_{m\to\infty}b_{n,m}=\min_{s\le n}\frac{\overline{a}_s}{s}$
exists for every $n$. So, we can apply Proposition \ref{doublesequence} and obtain
the desired result.

The property for lower Banach density is proved in an entirely similar manner.
\end{proof}

\medskip
\section{Tensor products as idempotents}

Given sets $I$ and $J$, consider the following operation $\star$ on $I\times J$:
$$(a,b)\star(c,d)=(a,d).$$

It is readily verified that $\star$ is associative, and so
$(I\times J,\star)$ is a semigroup. As already mentioned,
any semigroup has a canonical extension
to a semigroup on the corresponding space of ultrafilters (see \cite[Ch.4]{hs}).
In our case, one obtains a semigroup $(\beta(I\times J),{\footnotesize{\textcircled{$\star$}}})$ where
the operation ${\footnotesize{\textcircled{$\star$}}}$ between ultrafilters on $I\times J$ is defined by setting
for every $X\subseteq I\times J$:
$$X\in\V{\footnotesize{\textcircled{$\star$}}}\W\Leftrightarrow
\{(a,b)\in I\times J\mid \{(c,d)\in I\times J\mid (a,b)\star(c,d)\in X\}\in\W\}\in\V.$$

This operation ${\footnotesize{\textcircled{$\star$}}}$ is
closely related to tensor products.

\begin{proposition}
Let $\V,\W$ be ultrafilters on $I\times J$. Then 
$\V{\footnotesize{\textcircled{$\star$}}}\W=\pi_1(\V)\otimes\pi_2(\W)$,
where $\pi_1:I\times J\to I$ and $\pi_2:I\times J\to J$
are the canonical projections.
\end{proposition}

\begin{proof}
For every $X\subseteq I\times J$ the following are equivalent conditions:
\begin{itemize}
\item
$X\in \V{\footnotesize{\textcircled{$\star$}}}\W$
\item
$\{(a,b)\in I\times J\mid \{(c,d)\in I\times J\mid (a,b)\star(c,d)\in X\}\in\W\}\in\V$.
\item
$\{(a,b)\in I\times J\mid \{(c,d)\in I\times J\mid (a,d)\in X\}\in\W\}\in\V$.
\item
$\{(a,b)\in I\times J\mid I\times\{d\in J\mid (a,d)\in X\}\in\W\}\in\V$.
\item
$\{(a,b)\in I\times J\mid \pi_2^{-1}(\{d\in J\mid (a,d)\in X\})\in\W\}\in\V$.
\item
$\{(a,b)\in I\times J\mid \{d\in J\mid (a,d)\in X\}\in\pi_2(\W)\}\in\V$.
\item
$\{a\in I\mid\{d\in J\mid (a,d)\in X\}\in\pi_2(\W)\}\times J\in\V$.
\item
$\pi_1^{-1}(\{a\in I\mid\{d\in J\mid (a,d)\in X\}\in\pi_2(\W)\})\in\V$.
\item
$\{a\in I\mid\{d\in J\mid (a,d)\in X\}\in\pi_2(\W)\}\in\pi_1(\V)$.
\item
$X\in\pi_1(\V)\otimes\pi_2(\W)$.
\end{itemize}
\end{proof}

\begin{corollary}
An ultrafilter $\V$ on $I\times J$ is a tensor product if and only if 
$\V{\footnotesize{\textcircled{$\star$}}}\V=\V$ is idempotent with respect
to the operation ${\footnotesize{\textcircled{$\star$}}}$.
\end{corollary}

\begin{proof}
Recall that an ultrafilter $\V$ on $I\times J$
is a tensor product if and only if $\V=\pi_1(\V)\otimes\pi_2(\V)$,
and then apply the previous proposition.
\end{proof}

\medskip
\section{Combinatorial properties of the sets of tensor products}\label{tensor}

This section is dedicated to the following general combinatorial property of the sets 
that belong to tensor products. (It is an
extension of Theorems \ref{UotimesU} and \ref{UotimesV}.)

\begin{theorem}\label{main}
Let $\Phi$ be a nonempty set of functions
from $\{1,\ldots,k\}$ onto $\{1,\ldots,m\}$,
let $I_1,\ldots,I_m$ be infinite sets,
and for every $\varphi\in\Phi$ let $X^\varphi\subseteq I_{\varphi(1)}\times\ldots\times I_{\varphi(k)}$. 
Then the following properties are equivalent:
\begin{enumerate}
\item
There exist (distinct) non-principal ultrafilters $\U_\ell$ on $I_\ell$ for $\ell=1,\ldots,m$
such that for every $\varphi\in\Phi$ one has
$X^\varphi\in\U_{\varphi(1)}\otimes\cdots\otimes\U_{\varphi(k)}$.
\item
There exist sequences $(a_{\ell,n})_{n\in\N}$ in $I_\ell$ for $\ell=1,\ldots,m$ such that:
\begin{itemize}
\item[$(a)$]
$a_{\ell,n}\ne a_{\ell',n'}$ for $(\ell,n)\ne(\ell',n')$; \emph{i.e.}, every $(a_{\ell,n})_{n\in\N}$ is 1-1
and the ranges $\{a_{\ell,n}\mid n\in\N\}$ are pairwise disjoint sets.
\item[$(b)$]
If $I_\ell=\N$ we can take the sequence $(a_{\ell,n})_n$ to be increasing.
\item[$(c)$]
For every $\varphi\in\Phi$, and for 
every $(j_1,\ldots,j_m)\in\N^m$ where $j_1\le\ldots\le j_m$ 
and where $j_s<j_{s+1}$ whenever $\varphi(s)\ge\varphi(s+1)$, it is
$$(a_{\varphi(1),j_1},\ldots,a_{\varphi(k),j_k})\in X^\varphi.$$
\end{itemize}
\end{enumerate}
\end{theorem}

\begin{remark}
In property (2) it is taken
into account the possibility that the index sets are not disjoint;
for example, when $I_1=\ldots=I_k$ one has that (the ranges of the) sequences
$(a_{\ell,n})_{n\in\N}$ are pairwise disjoint.
\end{remark}

\begin{remark}
Without loss of generality, one can always assume that the considered 
non-principal ultrafilters $\U_1,\ldots,\U_m$ are distinct.
Indeed, otherwise one can consider $\{\V_1,\ldots,\V_{m'}\}=\{\U_1,\ldots,\U_m\}$
where the ultrafilters $\V_1,\ldots,\V_{m'}$ are distinct, 
and then replace every surjective function $\varphi:\{1,\ldots,k\}\to\{1,\ldots,m\}$
in $\Phi$ with the surjective function $\varphi':\{1,\ldots,k\}\to\{1,\ldots,m'\}$ where
$\varphi'(i)=\ell\Leftrightarrow \U_{\varphi(i)}=\V_\ell$.
\end{remark}

\begin{remark}\label{remark-varphi(1)}
By reordering the sets $I_1,\ldots,I_k$ if necessary, 
one can always assume that there exists $\varphi\in\Phi$ such that $\varphi(1)=1$.
Besides, when $k=m$, by reordering the sets $I_1,\ldots,I_k$ if necessary, 
one can even assume that the identity function $id\in\Phi$.
\end{remark}

The proof of Theorem \ref{main} is postponed to the last Section \ref{sec-proofs}.
We see below a few relevant particular cases, and some consequences about sum-sets.

\begin{example}\label{example-U^k}
Let $\Phi=\{c\}$ where $c:\{1,\ldots,k\}\to\{1\}$ is the constant function.
In this case the previous theorem reduces to Theorem \ref{UotimesU},
\emph{i.e.} for every $X\subseteq I^k$, there exist
a non-principal ultrafilter $\U$ on $I$ with $X\in\U^{k\otimes}$ 
if and only if there exists a 1-1 sequence 
$(b_n)$ in $I$ such that $(b_{j_1},\ldots,b_{j_k})\in X$ for all $j_1<\ldots<j_k$.
Moreover, if $I=\N$ we can take the sequence $(b_n)$ to be increasing.
\end{example}

As a straight corollary one obtains an ultrafilter characterization
of a sumsets property. Let us first introduce notation.

\smallskip
Recall that the \emph{pseudo-sum} $\U\oplus\V$ of ultrafilters on $\N$ is defined by setting 
$$A\in\U\oplus\V\Leftrightarrow\{n\in\N\mid A-n\in\V\}\in\U$$
where $A-n:=\{m\in\N\mid m+n\in A\}$ is the leftward shift of $A$ by $n$.

It is readily verified from the definitions that $\U\oplus\V=\text{Sum}(\U\otimes\V)$
is the image ultrafilter of the tensor product under the sum function
$\text{Sum}:(n,m)\mapsto n+m$, \emph{i.e.}
$A\in\U\oplus\V$ if and only if
the preimage 
$\text{Sum}^{-1}(A)=\{(m,n)\mid n+m\in A\}\in\U\otimes\V$.
The operation $\oplus$ is associative but not commutative
(verification is straightforward but not entirely trivial).
Accordingly, for every $k$ one has that
$A\in\U_1\oplus\cdots\oplus\U_k$ if and only if
$\text{Sum}_k^{-1}(A):=\{(n_1,\ldots,n_k)\in\N^k\mid n_1+\ldots+n_k\in A\}\in\U_1\otimes\cdots\otimes\U_k$.

\begin{notation}
Let $B_s=\{b_{s,1}<b_{s,2}<\ldots<b_{s,n}<b_{s,n+1}<\ldots\}\subseteq\N$ be infinite sets
for $s=1,\ldots,k$. We denote by
$$B_1\oplus\cdots\oplus B_k:=\{b_{1,j_1}+\ldots+b_{k,j_k}\mid j_1<\ldots<j_k\}.$$
When $B_1=\ldots=B_k$ we write $B^{k\oplus}$ to denote $B\oplus\ldots\oplus B$ ($k$ times).
\end{notation}

Note that:
$$B^{k\oplus}:=\{b_1+\ldots+b_k\mid b_1,\ldots,b_k\in B\ \text{distinct}\}=
\left\{\sum_{b\in F}b\ \Big|\ F\subset B,\ |F|=k\right\}.$$

\begin{proposition}\label{B^k}
Let $A\subseteq\N$ and let $k\in\N$. The following are equivalent:
\begin{enumerate}
\item
There exists an infinite set $B$ such that $B^{k\oplus}\subseteq A$.
\item
There exists a non-principal ultrafilter $\U$ on $\N$ with
$A\in\U^{k\oplus}$.
\end{enumerate}
\end{proposition}

\begin{proof}
Let $X:=\text{Sum}_k^{-1}(A)\subseteq\N^k$. 
According to Example \ref{example-U^k}, 
there exist a non-principal ultrafilter $\U$ on $\N$
with $X\in\U^{k\otimes}$ if and only if there exists an increasing sequence
$(b_n)$ such that $\{(b_{j_1},\ldots,b_{j_k})\mid j_1<\ldots<j_k\}\subseteq X$.
Set $B:=\{b_n\mid n\in\N\}$, and observe that $X\in\U^{k\otimes}\Leftrightarrow A\in\U^{k\oplus}$,
and that $\{(b_{j_1},\ldots,b_{j_k})\mid j_1<\ldots<j_k\}\subseteq X\Leftrightarrow
\{b_{j_1}+\ldots+b_{j_k}\mid j_1<\ldots<j_k\}=B^{k\oplus}\subseteq A$.
\end{proof}

\smallskip
The previous example can be extended to tensor products
of $k$ distinct ultrafilters.

\begin{example}\label{example-kmany}
Let $\Phi=\{id\}$ where $id:\{1,\ldots,k\}\to\{1,\ldots,k\}$ is
the identity function.
In this case Theorem \ref{main} reduces to Theorem \ref{UotimesV}, \emph{i.e.}
for every $X\subseteq I_1\times\ldots\times I_k$,
there exist distinct non-principal ultrafilters $\U_1,\ldots,\U_k$ 
on $I_1,\ldots,I_k$ respectively such that $X\in\U_1\otimes\cdots\otimes\U_k$
if and only if there exist disjoint 1-1 sequences 
$(a_{1,n})_n,\ldots,(a_{k,n})_n$ such that $(a_{1,j_1},\ldots,a_{k,j_k})\in X$ for all $j_1\le \ldots\le j_k$.
Moreover, if $I_1=\ldots=I_k=\N$ then we can take the sequences $(a_{s,n})_n$ to be increasing.
\end{example}

In the same way as Proposition \ref{B^k} follows from Example \ref{example-U^k},
the above example implies the following.

\begin{proposition}\label{kmany}
Let $A\subseteq\N$ and $k\in\N$. Then the following are equivalent:
\begin{enumerate}
\item
There exist disjoint infinite sets $B_s$ for $s=1,\ldots,k$
with $B_1\oplus\cdots\oplus B_k\subseteq A$.
\item
There exist distinct non-principal ultrafilters $\U_1,\ldots,\U_k$ on $\N$ 
such that $A\in\U_1\oplus\cdots\oplus\U_k$.
\end{enumerate}
\end{proposition}

Variants of the previous property can also be demonstrated in the case where
some of the ultrafilters $\U_s$ repeat (\emph{i.e.} $\U_s=\U_t$ for some $s<t$),
and some of the sets $B_s$ repeat accordingly
(\emph{i.e.} $B_s=B_t$ for some $s<t$).
(All such variants follows from Theorem \ref{main}.)

\smallskip
It is worth stressing that Theorem \ref{main} also allows us to characterize those sets $A$ that includes 
\emph{whole} $k$-sumsets 
$$B_1+\ldots+B_k:=\{b_1+\ldots+b_k\mid b_s\in B_s\ \text{for}\ s=1,\ldots,k\}.$$

Let us start with the basic case $k=2$.

\begin{example}\label{example-B+C}
Let $\Phi=\{\varphi,\psi\}$ where $\varphi,\psi:\{1,2\}\to\{1,2\}$
are the functions such that $\varphi(1)=\psi(2)=1$ and $\varphi(2)=\psi(1)=2$.
For every $X\subseteq I\times I$, if we let $X^\varphi=X^\psi=X$,
then Theorem \ref{main} states that there exist 
distinct non-principal ultrafilters $\U$ and $\V$ on $I$  
such that $X\in(\U\otimes\V)\cap(\V\otimes\U)$ if and only if there exist
disjoint 1-1 sequences $(b_n)$ and $(c_n)$ in $I$ such that
$$\{(b_i,c_j)\mid i\le j\}\cup\{(c_j,b_i)\mid j<i\}\subseteq X.$$
Moreover, if $I=\N$ then we can take the sequences
$(b_n)$ and $(c_n)$ to be increasing.
\end{example}

As a straight consequence we obtain the following characterization.

\begin{proposition}\label{B+C}
Let $A\subseteq\N$. The following properties are equivalent:
\begin{enumerate}
\item
There exist disjoint infinite sets $B,C$ with
$B+C\subseteq A$.
\item
There exist distinct non-principal ultrafilters $\U, \V$ on $\N$ with 
\\
$A\in(\U\oplus\V)\cap(\V\oplus\U)$.\footnote
{~This ultrafilter characterization has a small history.
It was first isolated by this author in April 2014 during 
a research program at the American Institute of Mathematics ``SQuaRE program" on 
\emph{Nonstandard Methods in Number Theory}, held in 2013, 2014 and 2015 in Palo Alto and San Jos\`e.
(The ``SQuaRE" research team consisted of 
Isaac Goldbring, Renling Jin, Steven Leth, Martino Lupini, Karl Mahlburg, and this author.)
It was then extended to sumsets of $k$-many sets by M. Lupini (see Proposition \ref{Bk}).
The idea was to use that ultrafilter characterization as a tool towards the following old conjecture of Erd\"os':
\begin{itemize}
\item
``$B+C$ conjecture":\ 
\emph{For every set $A\subseteq\N$ of positive asymptotic density
there exist infinite $B,C$ such that the sumset $B+C\subseteq A$".}
\end{itemize}
The ``SQuaRE" team found a partial solution to the conjecture \cite{dgjllm},
but never found a way to use Proposition \ref{B+C}, that remained unpublished.
In 2018, that property was first used in a model-theoretic context by
U. Andrews, G. Conan and I. Goldbring in \cite{acg},
and a proof was included in that paper as Proposition 3.1.
The same property was also recalled as Proposition 10 in the survey \cite{gb}.
Finally, in 2019, J. Moreira, F.K. Richter, and D. Robertson
eventually succeeded in proving the full Erd\"os' $B+C$ conjecture,
and they found it convenient to use the characterization of Proposition \ref{B+C}
as a preliminary step for their ergodic proof (see \cite[Lemma 2.1]{mrr}).}
\end{enumerate}
\end{proposition}

\begin{proof}
Let $X=\text{Sum}^{-1}(A)\subseteq\N\times\N$, let $(b_n)$ and $(c_n)$ be disjoint
increasing sequences in $\N$ as in Example \ref{example-B+C}, and let
$B:=\{b_n\mid n\in\N\}$ and $C:=\{c _n\mid n\in\N\}$. Then observe that
$A\in(\U\oplus\V)\cap(\V\oplus\U)$ if and only if 
$X\in(\U\otimes\V)\cap(\V\otimes\U)$ if and only if
$\{(b_i,c_j)\mid i\le j\}\cup\{(c_j,b_i)\mid j<i\}\subseteq X$ if and only if
$$B+C=\{b_i+c_j\mid i\le j\}\cup\{c_j+b_i\mid j<i\}\subseteq A.$$
\end{proof}

\begin{remark}
The result above also holds when $\U=\V$. Indeed,
in this case one can use Proposition \ref{kmany} to
obtain an increasing sequence $(b_n)$ such that $b_i+b_j\in A$ for all $i<j$,
and then let $B:=\{b_{2n}\mid n\in\N\}$ and $C:=\{b_{2n-1}\mid n\in\N\}$.
Clearly $B$ and $C$ are disjoint, and $B+C\subseteq A$, as desired.
\end{remark}

\smallskip
Also the generalisation to sums of $k$-many sets holds.
To this end we need first to isolate a consequence of Theorem \ref{main}
where one considers as functions the set of all permutations.

\smallskip
Let $\mathfrak{S}_k$ denote the set of all
permutations $\sigma:\{1,\ldots,k\}\to\{1,\ldots,k\}$.

\begin{definition}
We say that $\sigma\in\mathfrak{S}_k$ is a \emph{good ordering}
for $(j_1,\ldots,j_k)\in\N^k$ if
$j_{\sigma(1)}\le\ldots\le j_{\sigma(k)}$ and
 $j_{\sigma(s)}<j_{\sigma(s+1)}$ whenever $\sigma(s)>\sigma(s+1)$.
\end{definition}

\begin{lemma}
Every $(j_1,\ldots,j_k)\in\N^k$ admits a good ordering.
\end{lemma}

\begin{proof}
Inductively define $\sigma\in\mathfrak{S}_k$ by letting:
\begin{itemize}
\item
$\sigma(1)=\min\{t\mid j_t=\min\{j_1,\ldots,j_k\}\}$.
\item
$\sigma(s+1)=\min\{t\mid j_t=\min(\{j_1,\ldots,j_k\}\setminus\{j_{\sigma(1)},\ldots,j_{\sigma(s)}\})\}$ for $s<k$.
\end{itemize}

Clearly, $j_{\sigma(1)}\le\ldots\le j_{\sigma(k)}$. 
Now observe that if there are elements in
$\{j_1,\ldots,j_k\}\setminus\{j_{\sigma(1)}\le\ldots\le j_{\sigma(s)}\}$ that are equal to $j_{\sigma(s)}$
then, by the definition of $\tau$, we have $\sigma(s+1)>\sigma(s)$.
Otherwise, all elements in
$\{j_1,\ldots,j_k\}\setminus\{j_{\sigma(1)}\le\ldots\le j_{\sigma(s)}\}$ are greater
than $j_{\sigma(s)}$, and hence $j_{\sigma(s)}<j_{\sigma(s+1)}$.
Finally, notice that $\sigma(s)>\sigma(s+1)$ can occur only in the latter case,
and so $\sigma$ is a good ordering of $(j_1,\ldots,j_k)$. 
\end{proof}

A useful consequence of the previous theorem is the following:

\begin{proposition}\label{permutations}
For every $X\subseteq\N^k$, the following properties are equivalent:
\begin{enumerate}
\item
There exist (distinct) non-principal ultrafilters $\U_1,\ldots,\U_k$ on $\N$
such that $X\in\U_{\sigma(1)}\otimes\cdots\otimes\U_{\sigma(k)}$
for every $\sigma\in\mathfrak{S}_k$.
\item
There exist disjoint increasing sequences $(a_{\ell,n})_{n\in\N}$ for $\ell=1,\ldots,k$ such that
for every $\sigma\in\mathfrak{S}_k$ and for 
every $(j_1,\ldots,j_k)\in\N^k$ where $j_1\le\ldots\le j_k$ 
and where $j_s<j_{s+1}$ whenever $\sigma(s)>\sigma(s+1)$, it is 
$(a_{\sigma(1),j_1},\ldots,a_{\sigma(k),j_k})\in X$.
\item
There exist disjoint increasing sequences $(a_{\ell,n})_{n\in\N}$
for $\ell=1,\ldots,k$ such that for every $(j_1,\ldots,j_k)\in\N^k$ 
and for every good ordering $\tau$ of $(j_1,\ldots,j_k)$, it is
$(a_{\tau(1),j_{\tau(1)}},\ldots,a_{\tau(k),j_{\tau(k)}})\in X$.
\end{enumerate}
\end{proposition}

\begin{proof}
$(1)\Leftrightarrow(2)$ is the particular case
of Theorem \ref{main} where $\Phi=\mathfrak{S}_k$ and 
where $X^\sigma=X$ for all $\sigma\in\mathfrak{S}_k$. 

\smallskip
$(2)\Rightarrow(3)$. By the definition of good ordering,
the tuple $(j_{\tau(1)},\ldots,j_{\tau(k)})$ satisfies the hypotheses of condition $(2)$
for the permutation $\tau$, and so we can conclude that
$(a_{\tau(1),j_{\tau(1)}},\ldots,a_{\tau(k),j_{\tau(k)}})\in X$.

\smallskip
$(3)\Rightarrow(2)$.
Suppose that $\sigma\in\mathfrak{S}_k$ and $(j_1,\ldots,j_k)\in\N^k$ 
are such that $j_1\le\ldots\le j_k$ and $j_s<j_{s+1}$ whenever $\sigma(s)>\sigma(s+1)$.
Take the inverse permutation $\theta=\sigma^{-1}$, and 
let $j'_s:=j_{\theta(s)}$ for $s=1,\ldots,k$. 
We observe that $j'_{\sigma(s)}=j_{\sigma(\theta(s))}=j_s$
for every $s=1,\ldots,k$, and hence $\sigma$ is a good
ordering of $(j'_1,\ldots,j'_k)$. Then we can conclude that
$(a_{\sigma(1),j_1},\ldots,a_{\sigma(k),j_k})=
(a_{\sigma(1),j'_{\sigma(1)}},\ldots,a_{\sigma(k),j'_{\sigma(k)}})\in X$, as desired.
\end{proof}

We are now ready to prove the following general property about sum-sets.

\begin{proposition}\label{Bk}
Let $A\subseteq\N$ and let $k\in\N$. The following properties are equivalent:
\begin{enumerate}
\item
There exist distinct non-principal ultrafilters $\U_1,\ldots,\U_k$ on $\N$ such that 
\\
$A\in\U_{\sigma(1)}\oplus\cdots\oplus\U_{\sigma(k)}$ for every permutation $\sigma$ of $\{1,\ldots,k\}$.
\item
There exist disjoint infinite sets $B_1,\ldots,B_k$ with
$B_1+\ldots+B_k\subseteq A$.
\end{enumerate}
\end{proposition}

\begin{proof}
Consider the preimage $\text{Sum}_k^{-1}(A)=\{(n_1,\ldots,n_k)\in\N^k\mid n_1+\ldots+n_k\in A\}$.
By the previous Proposition \ref{permutations}, the following conditions are equivalent:
\begin{itemize}
\item[$(i)$]
There exists (distinct) non-principal ultrafilters $\U_1,\ldots,\U_k$ on $\N$
such that $\text{Sum}_k^{-1}(A)\in\U_{\sigma(1)}\otimes\cdots\otimes\U_{\sigma(k)}$
for every $\sigma\in\mathfrak{S}_k$.
\item[$(ii)$]
There exist disjoint increasing sequences $(b_{\ell,n})_{n\in\N}$
for $\ell=1,\ldots,k$ such that for every $(j_1,\ldots,j_k)\in\N^k$ 
and for every good ordering $\sigma$ of $(j_1,\ldots,j_k)$, it is
$(a_{\sigma(1),j_{\sigma(1)}},\ldots,a_{\sigma(k),j_{\sigma(k)}})\in\text{Sum}_k^{-1}(A)$.
\end{itemize}
Now observe that, by the definitions, 
$$\text{Sum}^{-1}_k(A)\in\U_{\sigma(1)}\otimes\cdots\otimes\U_{\sigma(k)}\Longleftrightarrow 
A\in\U_{\sigma(1)}\oplus\cdots\oplus\U_{\sigma(k)},$$
and so condition $(i)$ is equivalent to condition $(1)$.

We also observe that $(b_{\sigma(1),j_{\sigma(1)}},\ldots,b_{\sigma(k),j_{\sigma(k)}})\in\text{Sum}_k^{-1}(A)$
for every good ordering $\sigma$ of $(j_1,\ldots,j_k)$, if and only if
$b_{1,j_1}+\ldots+b_{k,j_k}\in A$. 
So, if we define the pairwise disjoint sets $B_\ell:=\{a_{\ell,n}\mid n\in\N\}$ for $\ell=1,\ldots,k$,
then condition $(ii)$ is equivalent to condition $(2)$, and this concludes the proof.
\end{proof}

Theorem \ref{main} has even more general consequences,
in that it allows considering sum-sets where some of the sets may be repeated.
The next simple example is given illustrate the property; other more intricate examples
can also be obtained in the same fashion.

\begin{example}\label{example-B+B+C}
Let $\Phi$ be the set of all surjective functions $\varphi:\{1,2,3\}\to\{1,2\}$
such that the preimage $\varphi^{-1}(\{1\})$ contains two elements, and 
the preimage $\varphi^{-1}(\{2\})$ contains one element.
Given $X\subseteq I^3$, if we let $X^\varphi=X$ for all $\varphi\in\Phi$,
then by Theorem \ref{main} we obtain the equivalence of the following two conditions:
\begin{itemize}
\item 
There exist distinct non-principal ultrafilters $\U$ and $\V$ on $I$ such that 
$$X\in(\U\otimes\U\otimes\V)\cap(\U\otimes\V\otimes\U)\cap(\V\otimes\U\otimes\U).$$
\item
There exist disjoint 1-1 sequences $(b_n)$ and $(c_n)$ in $I$ such that
\begin{multline*}
\hspace{1cm}
\{(b_{j_1},b_{j_2},c_{j_3})\mid j_1<j_2\le j_3\} \cup
\\
\cup \{(b_{j_1},c_{j_3},b_{j_2})\mid j_1\le j_3< j_2\}\cup
\{(c_{j_3},b_{j_1},b_{j_2})\mid j_3<j_1<j_2\}\subseteq X.
\end{multline*}
\end{itemize}
Moreover, if $I=\N$ then we can take the sequences
$(b_n)$ and $(c_n)$ to be increasing.
\end{example}

As a consequence of the above example, we obtain:

\begin{proposition}\label{B+B+C}
Let $A\subseteq\N$. The following properties are equivalent:
\begin{enumerate}
\item
There exist disjoint infinite sets $B$ and $C$
with $(B\oplus B)+C\subseteq A$.
\item
There exist distinct non-principal ultrafilters $\U$ and $\V$ on $\N$ such that 
\\
$A\in(\U\oplus\U\oplus\V)\cap(\U\oplus\V\oplus\U)\cap(\V\oplus\U\oplus\U)$.
\end{enumerate}
\end{proposition}

\begin{proof}
Let $X:=\text{Sum}_3^{-1}(A)\subseteq\N^3$. By Example \ref{example-B+B+C}
there exist non-principal ultrafilters $\U$ and $\V$ on $\N$
such that $X\in(\U\otimes\U\otimes\V)\cap(\U\otimes\V\otimes\U)\cap(\V\otimes\U\otimes\U)$
if and only if there exist disjoint increasing sequences $(b_n)$ and $(c_n)$ with 
\begin{enumerate}
\item[$(i)$]
$b_{j_1}+b_{j_2}+c_{j_3}\in A$ for all $j_1<j_2\le j_3$;
\item[$(ii)$]
$b_{j_1}+c_{j_3}+b_{j_2}\in A$ for all $j_1\le j_3< j_2$;
\item[$(iii)$]
$c_{j_3}+b_{j_1}+b_{j_2}\in A$ for all $j_3<j_1<j_2$.
\end{enumerate}
Let $B:=\{b_n\mid n\in\N\}$ and $C:=\{c_n\mid n\in\N\}$.  
By the definitions, $X\in(\U\otimes\U\otimes\V)\cap(\U\otimes\V\otimes\V)\cap(\V\otimes\U\otimes\U)$
if and only if $A\in(\U\oplus\U\oplus\V)\cap(\U\oplus\V\oplus\V)\cap(\V\oplus\U\oplus\U)$;
and $(x,y,z)\in X\Leftrightarrow x+y+z\in A$.
Finally, it can be verified in a straightforward manner
that $(B\oplus B)+C:=\{b_{j_1}+b_{j_2}+c_{j_3}\mid j_1<j_2\}\subseteq A$ if and only if
the above three conditions $(i),(ii),(iii)$ hold.
\end{proof}

Let us now formulate the general version of the previous property.

\begin{theorem}\label{generalsumsets}
Let $A\subseteq\N$ and let $n_1,\ldots,n_k\in\N$. The following properties are equivalent:
\begin{enumerate}
\item
There exist disjoint infinite sets $B_1,\ldots,B_k$ such that
$$(B_1)^{n_1\oplus}+\ldots+(B_k)^{n_k\oplus}\subseteq A.$$
\item
There exist distinct non-principal ultrafilters $\U_1,\ldots,\U_k$ on $\N$ such that 
\\
$$A\in\bigcap_{\varphi\in\Phi}\U_{\varphi(1)}\oplus\cdots\oplus\U_{\varphi(n)}$$
where $n=n_1+\ldots+n_k$, and $\Phi$ is the set of all (surjective) functions 
$\varphi:\{1,\ldots,n\}\to\{1,\ldots,k\}$
where for every $s=1,\ldots,k$ the preimage $\varphi^{-1}(\{s\})$ contains $n_s$ elements.
\end{enumerate}
\end{theorem}

The proof is obtained in the same way as in Proposition \ref{B+B+C}, 
only with heavier notation, and is omitted.

\smallskip
We remark that a variety of even more general properties can be obtained as consequences of Theorem \ref{main};
the following is just one of those.

\begin{example}\label{example-B+C+C+D+D}
Let $\Phi$ be the set of all surjective functions $\varphi:\{1,2,3,4,5\}\to\{1,2,3\}$
such that the preimage $\varphi^{-1}(\{1\})=\{1\}$, and both preimages 
$\varphi^{-1}(\{2\})$ and $\varphi^{-1}(\{3\})$ contain two elements.
Given $X\subseteq I^5$, if we let $X^\varphi=X$ for all $\varphi\in\Phi$,
by Theorem \ref{main} we obtain that the following two conditions are equivalent:
\begin{itemize}
\item
There exist distinct non-principal ultrafilters $\U, \V, \W$ on $I$ such that 
$X\in(\U\otimes\V\otimes\V\otimes\W\otimes\W)\cap(\U\otimes\V\otimes\W\otimes\V\otimes\W)\cap
(\U\otimes\V\otimes\W\otimes\W\otimes\V)\cap(\U\otimes\W\otimes\V\otimes\V\otimes\W)\cap
(\U\otimes\W\otimes\V\otimes\W\otimes\V)\cap(\U\otimes\W\otimes\W\otimes\V\otimes\V)$
\item
There exist
disjoint 1-1 sequences $(b_n), (c_n), (d_n)$ in $I$ such that
\begin{multline*}
\hspace{1cm}
\{(b_{j_1},c_{j_2},c_{j_3},d_{j_4},d_{j_5})\mid j_1\le j_2<j_3\le j_4<j_5\}\cup
\\
\cup \{(b_{j_1},c_{j_3},d_{j_4},c_{j_3},d_{j_5})\mid j_1\le j_3\le j_4<j_3\le j_5\}\cup
\\
\cup \{(b_{j_1},c_{j_2},d_{j_4},d_{j_5},c_{j_3})\mid j_1\le j_2\le j_4<j_5<j_3\}\cup
\\
\cup \{(b_{j_1},d_{j_4},c_{j_2},c_{j_3},d_{j_5})\mid j_1\le j_4< j_2<j_3\le j_5\}\cup
\\
\cup \{(b_{j_1},d_{j_4},c_{j_2},d_{j_5},c_{j_3})\mid j_1\le j_4< j_2\le j_5<j_2\}\cup
\\
\cup \{(b_{j_1},d_{j_4},d_{j_5},c_{j_2},c_{j_3})\mid j_1\le j_4<j_5< j_2<j_3\}\subseteq X.
\end{multline*}
\end{itemize}
Moreover, if $I=\N$ then we can take the sequences
$(b_n), (c_n), (d_n)$ to be increasing.
\end{example}

As a consequence, we obtain:

\begin{proposition}\label{B+C+C+D+D}
Let $A\subseteq\N$. The following properties are equivalent:
\begin{enumerate}
\item
There exist disjoint infinite sets $B, C, D$ with 
\begin{multline*}
\hspace{1cm}\{b+c+c'+d+d'\mid 
b\in B\ \text{and}\ c, c'\in C\ \text{and}\ d, d'\in D 
\\
\text{and}\ c\ne c'\ \text{and}\ d\ne d'\ \text{and}\ b<c, c', d, d'\}\subseteq A.
\end{multline*}
\item
There exist distinct non-principal ultrafilters $\U, \V, \W$ on $\N$ such that 

\smallskip
\noindent
{\small{$A\in(\U\oplus\V\oplus\V\oplus\W\oplus\W)\cap(\U\oplus\V\oplus\W\oplus\V\oplus\W)\cap
(\U\oplus\V\oplus\W\otimes\W\oplus\V)\cap(\U\oplus\W\oplus\V\oplus\V\oplus\W)\cap
(\U\oplus\W\oplus\V\oplus\W\oplus\V)\cap(\U\oplus\W\oplus\W\oplus\V\oplus\V)$.}}
\end{enumerate}
\end{proposition}

\begin{proof}
Apply the property of Example \ref{example-B+C+C+D+D} to the set $X:=\text{Sum}_5^{-1}(A)\subseteq\N^5$.
Then there exist distinct non-principal ultrafilters $\U, \V, \W$ on $\N$ such that 
$A$ satisfies condition (2), if and only if there exist disjoint
increasing sequences $(b_n), (c_n), (d_n)$ with 
\begin{enumerate}
\item[$(i)$]
$b_{j_1}+c_{j_2}+c_{j_3}+d_{j_4}+d_{j_5}\in A$ for all $j_1\le j_2<j_3\le j_4<j_5$;
\item[$(ii)$]
$b_{j_1}+c_{j_2}+d_{j_4}+c_{j_3}+d_{j_5}\in A$ for all $j_1\le j_2\le j_4<j_3\le j_5$;
\item[$(iii)$]
$b_{j_1}+c_{j_2}+d_{j_4}+d_{j_5}+c_{j_3}\in A$ for all $j_1\le j_2\le j_4<j_5<j_3$;
\item[$(iv)$]
$b_{j_1}+d_{j_4}+c_{j_2}+c_{j_3}+d_{j_5}\in A$ for all $j_1\le j_4< j_2<j_3\le j_5$;
\item[$(v)$]
$b_{j_1}+d_{j_4}+c_{j_2}+d_{j_5}+c_{j_3}\in A$ for all $j_1\le j_4<j_2\le j_5<j_3$;
\item[$(vi)$]
$b_{j_1}+d_{j_4}+d_{j_5}+c_{j_2}+c_{j_3}\in A$ for all $j_1\le j_4<j_5< j_2<j_3$.
\end{enumerate}

Let $B:=\{b_n\mid n\in\N\}$, $C:=\{c_n\mid n\in\N\}$, and $D:=\{d_n\mid n\in\N\}$.
By removing some elements from $C$ and $D$ if necessary, we can assume that
$b_n<c_n,d_n$ for every $n$, so that the six conditions above
are equivalent to the desired property (1). 
\end{proof}

\begin{remark}
The characterisations presented above about sum-sets and pseudo-sums between ultrafilters
also holds for ``product-sets" 
and ``pseudo-products" between ultrafilters.
Indeed, for infinite sets $A=\{a_n\mid n\in\N\}$ and $B=\{b_n\mid n\in\N\}$
where elements are arranged in increasing order, one define their product-set
by letting $A\odot B=\{a_n\cdot b_m\mid n<m\}$.
Besides, in the same way as $\U\oplus\V=\text{Sum}(\U\otimes\V)$ is defined as the image 
of the tensor product $\U\otimes\V$
under the sum map $\text{Sum}:(n,m)\mapsto n+m$, one defines
the pseudo-product $\U\odot\V=\text{Prod}(\U\otimes\V)$ 
as the image of the tensor product $\U\otimes\V$
under the product map $\text{Prod}:(n,m)\mapsto n\cdot m$.
Equivalently, the pseudo-product is defined by setting
$A\in\U\odot\V\Leftrightarrow\{n\in\N\mid A/n\in\V\}\in\U$,
where $A/n=\{m\in\N\mid m\cdot n \in A\}$.

More generally, it is well known that one can even define
a canonical extension of any semigroup $(S,\star)$ to a semigroup 
$(\beta S,{\footnotesize{\textcircled{$\star$}}})$ on the space of ultrafilters,
by letting
$\U{\footnotesize{\textcircled{$\star$}}}\V=\text{Star}(\U\otimes\V)$ be the ultrafilter image of the tensor product
$\U\otimes\V$ under the map $\text{Star}:(n,m)\mapsto n\star m$.\footnote
{~This means that $\U{\footnotesize{\textcircled{$\star$}}}\V$ is the ultrafilter on $S\times S$
such that, for every $X\subseteq S\times S$:
$$X\in\U{\footnotesize{\textcircled{$\star$}}}\V\Leftrightarrow
\{s\in S\mid\{t\in S\mid s\star t\in X\}\in\V\}\in\U.$$}
(See \cite[Ch.4]{hs} for a detailed discussion of this general case.)
We observe that when the operation $\star$ is commutative, one
obtains similar characterizations as the ones we presented above where one considers
``star-sets" $A{\footnotesize{\textcircled{$\star$}}}B=\{a_n\star b_m\mid n<m\}$
and ``pseudo-stars" $\U{\footnotesize{\textcircled{$\star$}}}\V$ between ultrafilters
on $S$.
\end{remark}

\medskip
\section{Proof of Theorem \ref{main}}\label{sec-proofs}

Before proving Theorem \ref{main} in its generality, let us first consider 
a special case that does not require heavy notation, so as to give a 
more transparent idea of the construction.

\begin{example}
Let $\U,\V$ be non-principal ultrafilters on $\N$ and assume that
\begin{itemize}
\item
$X^\varphi\in\U\otimes\U\otimes\V$;
\item
$X^\chi\in \U\otimes\V\otimes\U$; 
\item
$X^\psi\in\U\otimes\V\otimes\V$.
\item
$X^\omega\in\V\otimes\U\otimes\U$.
\end{itemize}
Then there exists 1-1 sequences $(a_n)$, $(b_n)$ with disjoint ranges such that:\footnote{
~Clearly, this corollary is one implication in the particular case of Theorem \ref{main} where 
$k=3$, $m=2$, $\U_1=\U$, $\U_2=\V$, and
$\Phi=\{\varphi,\chi,\psi,\omega\}$ where the functions
$\varphi,\chi,\psi,\omega:\{1,2,3\}\to\{1,2\}$ are such that
\begin{itemize}
\item
$\varphi(1)=\varphi(2)=1$ and $\varphi(3)=2$.
\item
$\chi(1)=\chi(3)=1$ and $\chi(2)=2$.
\item
$\psi(1)=1$ and $\psi(2)=\psi(3)=2$.
\item
$\omega(1)=2$ and $\omega(2)=\omega(3)=1$.
\end{itemize}}
\begin{enumerate}
\item
$(a_{j_1},a_{j_2},b_{j_3})\in X^\varphi$ for all $j_1<j_2\le j_3$.
\item
$(a_{j_1},b_{j_2},a_{j_3})\in X^\chi$ for all $j_1\le j_2<j_3$.
\item
$(a_{j_1},b_{j_2},b_{j_3})\in X^\psi$ for all $j_1\le j_2<j_3$.
\item
$(b_{j_1},a_{j_2},a_{j_3})\in X^\omega$ for all $j_1<j_2< j_3$.
\end{enumerate}
\end{example}

Let us outline the proof of the above property.
For $X\subseteq \N^3$, elements $a,a'\in\N$, and ultrafilters $\Xc,\Yc,\Zc$ on $\N$, denote by
\begin{itemize}
\item
$X_a:=\{(a',a'')\in \N^2\mid (a,a',a'')\in X\}$.
\item
$X_{a,a'}=\{a''\in\N\mid (a,a',a'')\in X\}$.
\item
$(X_a)_\Zc:=\{a'\in \N\mid X_{a,a'}\in\Zc\}$.
\item
$X_{\Yc,\Zc}=
\{a\in\N\mid (X_a)_\Zc\in\Yc\}$.
\end{itemize}

Note that

\begin{itemize}
\item
$X_{\Yc,\Zc}\in\Xc\Leftrightarrow X\in\Xc\otimes\Yc\otimes\Zc$.
\end{itemize}

\smallskip
We proceed by recursion on $n\in\N$, and 
define elements $a_n, b_n\in\N$ 
so that the following properties are satisfied:

\begin{enumerate}
\item
$b_n>a_n>b_{n-1}$.
\item
$a_{j_1}\in X^\varphi_{\U,\V}$, $a_{j_2}\in (X^\varphi_{a_{j_1}})_\V$, and
$(a_{j_1},a_{j_2},b_{j_3})\in X^\varphi$ for all $j_1<j_2\le j_3\le n$.
\item
$a_{j_1}\in X^\chi_{\V,\U}$, $b_{j_2}\in (X^\chi_{a_{j_1}})_\U$, and
$(a_{j_1},b_{j_2},a_{j_3})\in X^\chi$ for all $j_1\le j_2<j_3\le n$.
\item
$a_{j_1}\in X^\psi_{\V,\V}$, $b_{j_2}\in (X^\psi_{a_{j_1}})_\V$, and
$(a_{j_1},b_{j_2},b_{j_3})\in X^\psi$ for all $j_1\le j_2< j_3\le n$.
\item
$b_{j_1}\in X^\omega_{\U,\U}$, $a_{j_2}\in (X^\omega_{b_{j_1}})_\U$, and
$(b_{j_1},a_{j_2},a_{j_3})\in X^\omega$ for all $j_1<j_2<j_3\le n$.
\end{enumerate}

It is readily seen that the above properties yield the desired result.

\smallskip
Let us start with the base step $n=1$.
By the hypotheses,
\begin{itemize}
\item
$X^\varphi_{\U,\V}\in\U$, $X^\chi_{\V,\U}\in\U$, $X^\psi_{\V,\V}\in\U$,
$X^\omega_{\U,\U}\in\V$.
\end{itemize}
So, we can pick an element 
$a_1\in X^\varphi_{\U,\V}\cap X^\chi_{\V,\U}\cap X^\psi_{\V,\V}\in\U$.
Then
\begin{itemize}
\item
$(X^\varphi_{a_1})_\V\in\U$, $(X^\chi_{a_1})_\U\in\V$,
$(X^\psi_{a_1})_\V\in\V$.
\end{itemize}
Pick an element $b_1\in X^\omega_{\U,\U}\cap(X^\chi_{a_1})_\U\cap(X^\psi_{a_1})_\V\in\V$ 
with $b_1>a_1$. Then
\begin{itemize}
\item
$(X^\omega_{b_1})_\U\in\U$, $X^\chi_{a_1,b_1}\in\U$, $X^\psi_{a_1,b_1}\in\V$.
\end{itemize}
It is readily seen that $a_1, b_1$ satisfy the desired properties.

\smallskip
Let us turn to the step $n=2$. Pick an element 
$$a_2\in X^\varphi_{\U,\V}\cap X^\chi_{\V,\U}\cap X^\psi_{\V,\V}\cap
(X^\varphi_{a_1})_\V\cap(X^\omega_{b_1})_\U\cap X^\chi_{a_1,b_1}\in\U.$$
with $a_2>b_1$. (This is possible because all sets that belong to the non-principal ultrafilter
$\U$ are infinite). Then
\begin{itemize}
\item
$(X^\varphi_{a_2})_\V\in\U$, $(X^\chi_{a_2})_\U\in\V$,
$(X^\psi_{a_2})_\V\in\V$.
\item
$X^\varphi_{a_1,a_2}\in\V$, $X^\omega_{b_1,a_2}\in\U$.
\item
$(a_1,b_1,a_2)\in X^\chi$.
\end{itemize}
Pick an element 
$$b_2\in X^\omega_{\U,\U}\cap(X^\chi_{a_1})_\U\cap(X^\psi_{a_1})_\V\cap
(X^\chi_{a_2})_\U\cap(X^\psi_{a_2})_\V\cap X^\psi_{a_1,b_1}\cap X^\varphi_{a_1,a_2}\in\V$$
with $b_2>a_2$. (This is possible because all sets that belong to the non-principal ultrafilter
$\V$ are infinite). Then
\begin{itemize}
\item
$(X^\omega_{b_2})_\U\in\U$.
\item
$X^\chi_{a_1,b_2}\in\U$, $X^\psi_{a_1,b_2}\in\V$, 
$X^\chi_{a_2,b_2}\in\U$, $X^\psi_{a_2,b_2}\in\V$.
\item
$(a_1,b_1,b_2)\in X^\psi$, $(a_1,a_2,b_2)\in X^\varphi$.
\end{itemize}
Let us continue one more step to better illustrate the process.
Pick an element 
$$a_3\in X^\varphi_{\U,\V}\cap X^\chi_{\V,\U}\cap X^\psi_{\V,\V}\cap
(X^\varphi_{a_1})_\V\cap(X^\omega_{b_1})_\U\cap (X^\omega_{b_2})_\U\cap
X^\chi_{a_1,b_1}\cap X^\chi_{a_1,b_2}\cap X^\chi_{a_2,b_2}\in\U$$
with $a_3>b_2$. Then
\begin{itemize}
\item
$(X^\varphi_{a_3})_\V\in\U$, $(X^\chi_{a_3})_\U\in\V$,
$(X^\psi_{a_3})_\V\in\V$.
\item
$X^\varphi_{a_1,a_3}\in\V$, $X^\omega_{b_1,a_3}\in\U$,
$X^\omega_{b_2,a_3}\in\U$.
\item
$(a_1,b_1,a_3)\in X^\chi$, $(a_1,b_2,a_3)\in X^\chi$, $(a_2,b_2,a_3)\in X^\chi$.
\end{itemize}
Then pick an element 
\begin{multline*}
b_3\in X^\omega_{\U,\U}\cap(X^\chi_{a_1})_\U\cap(X^\psi_{a_1})_\V\cap
(X^\chi_{a_2})_\U\cap(X^\psi_{a_2})_\V\cap 
(X^\chi_{a_3})_\U\cap (X^\psi_{a_3})_\V\cap
\\
\cap X^\psi_{a_1,b_1}\cap X^\varphi_{a_1,a_2}\cap
X^\psi_{a_1,b_2}\cap X^\psi_{a_2,b_2}\cap X^\varphi_{a_1,a_3}\in\V
\end{multline*}
with $b_3>a_3$. Then
\begin{itemize}
\item
$(X^\omega_{b_3})_\U\in\U$.
\item
$X^\chi_{a_1,b_3}\in\U$, $X^\psi_{a_1,b_3}\in\V$, 
$X^\chi_{a_2,b_3}\in\U$, $X^\psi_{a_2,b_3}\in\V$,
$X^\chi_{a_3,b_3}\in \U$, $X^\psi_{a_3,b_3}\in\V$.
\item
$(a_1,b_1,b_3)\in X^\psi$, $(a_1,a_2,b_3)\in X^\varphi$, 
$(a_1,b_2,b_3)\in X^\psi$, $(a_2,b_2,b_3)\in X^\psi$, $(a_1,a_3,b_3)\in X^\varphi$.
\end{itemize}

By iterating the process, one obtains the sequences $(a_n)$ and $(b_n)$ with the desired properties.

\smallskip
Let us now see the complete proof of the general case.

\medskip
\noindent
\textbf{Theorem \ref{main}}
\emph {Let $\Phi$ be a nonempty set of functions
from $\{1,\ldots,k\}$ onto $\{1,\ldots,m\}$,
let $I_1,\ldots,I_m$ be infinite sets,
and for every $\varphi\in\Phi$ let $X^\varphi\subseteq I_{\varphi(1)}\times\ldots\times I_{\varphi(k)}$. 
Then the following properties are equivalent:
\begin{enumerate}
\item
There exist (distinct) non-principal ultrafilters $\U_\ell$ on $I_\ell$ for $\ell=1,\ldots,m$
such that for every $\varphi\in\Phi$ one has
$X^\varphi\in\U_{\varphi(1)}\otimes\cdots\otimes\U_{\varphi(k)}$
\item
There exist sequences $(a_{\ell,n})_{n\in\N}$ in $I_\ell$ for $\ell=1,\ldots,m$ such that:
\begin{itemize}
\item[$(a)$]
$a_{\ell,n}\ne a_{\ell',n'}$ for $(\ell,n)\ne(\ell',n')$; \emph{i.e.}, every $(a_{\ell,n})_{n\in\N}$ is 1-1
and the ranges $\{a_{\ell,n}\mid n\in\N\}$ are pairwise disjoint sets.
\item[$(b)$]
If $I_\ell=\N$ we can take the sequence $(a_{\ell,n})_n$ to be increasing.
\item[$(c)$]
For every $\varphi\in\Phi$, and for 
every $(j_1,\ldots,j_m)\in\N^m$ where $j_1\le\ldots\le j_m$ 
and where $j_s<j_{s+1}$ whenever $\varphi(s)\ge\varphi(s+1)$, it is
$$(a_{\varphi(1),j_1},\ldots,a_{\varphi(k),j_k})\in X^\varphi.$$
\end{itemize}
\end{enumerate}}

\begin{proof}[\textbf{Proof of Theorem \ref{main}}]

$(2)\Rightarrow(1)$. 
Pick any non-principal ultrafilter $\V$ on $\N$.
For every $\ell=1,\ldots,k$, let $\theta_\ell:\N\to I_\ell$ be the
function where $\theta_\ell:n\mapsto a_{\ell,n}$, and let
$\U_\ell:=\theta_\ell(\V)$ be the image ultrafilter on $I_\ell$.
Note that, since $\V$ is non-principal and $\theta_\ell$ is 1-1, then $\U_\ell$ is also non-principal.
We observe that for every $\varphi\in\Phi$, we have:
$$\Lambda^\varphi:=\{(a_{\varphi(1),j_1},\ldots,a_{\varphi(k),j_k})\in I_{\varphi(1)}\times\ldots\times I_{\varphi(k)}
\mid j_1<\ldots<j_k\}\in
\U_{\varphi(1)}\otimes\ldots\otimes\U_{\varphi(k)}.$$
Indeed, by using property (2) of Proposition \ref{tensorproperties} and Proposition \ref{superdiagonal}, we have
\begin{multline*}
\Lambda^\varphi\in\U_{\varphi(1)}\otimes\ldots\otimes\U_{\varphi(k)}=
(\theta_{\varphi(1)},\ldots,\theta_{\varphi(k)})(\V\otimes\ldots\otimes\V)\Leftrightarrow
\\
\Leftrightarrow (\theta_{\varphi(1)},\ldots,\theta_{\varphi(k)})^{-1}(\Lambda^\varphi)=
\{(j_1,\ldots,j_k)\mid j_1<\ldots<j_k\}=
\Delta^+_k\in \V\otimes\ldots\otimes\V.
\end{multline*}
Since by the hypothesis $\Lambda^\varphi\subseteq X^\varphi$,
we then obtain $X^\varphi\in\U_{\varphi(1)}\otimes\ldots\otimes\U_{\varphi(k)}$,
as desired.

\smallskip
$(1)\Rightarrow(2)$.  
Let us first fix notation.
Let $X\subseteq I_1\times\ldots\times I_m$,
and let $\V_1,\ldots\V_m$ be ultrafilters on the sets $I_1,\ldots,I_m$, respectively.
For every $(b_1,\ldots,b_m)\in I_1\times\ldots\times I_m$ and for every $s<m$, we denote:
\begin{itemize}
\item
$X_{b_1,\ldots,b_s}:=\{(b_{s+1},\ldots,b_m)\in I_{s+1}\times\ldots\times I_m\mid
(b_1,\ldots,b_s,b_{s+1},\ldots,b_m)\in X\}$.
\item
$(X_{b_1,\ldots,b_{s-1}})_{\V_{s+1},\ldots,\V_m}:=\{b\in I_s\mid
X_{b_1,\ldots,b_{s-1},b}\in\V_{s+1}\otimes\ldots\otimes\V_m\}$.
\end{itemize}

By definition of tensor product, we have
\begin{itemize}
\item
$X_{b_1,\ldots,b_{s-1}}\in\V_s\otimes\V_{s+1}\otimes\ldots\otimes\V_m\Longleftrightarrow
(X_{b_1,\ldots,b_{s-1}})_{\V_{s+1},\ldots,\V_m}\in\V_s$.
\end{itemize}

\smallskip
We proceed by recursion on $n\in\N$, and for fixed $n$ by recursion on $\ell=1,\ldots,m$,
and define elements $a_{\ell,n}\in I_\ell$ 
so that the following properties are satisfied:

\begin{enumerate}
\item[$(1)_n$]
$a_{\ell,n}\ne a_{\ell',n'}$ for all $\ell,\ell'$ when $n'<n$, 
and for all $\ell'<\ell$ when $n=n'$. 
\item[$(2)_n$]
If $I_\ell=\N$ then $a_{\ell,n}>a_{\ell,n-1}$.
\item[$(3)_n$]
$a_{\varphi(s),j_s}\in(X^\varphi_{a_{\varphi(1),j_1},\ldots,a_{\varphi(s-1),j_{s-1}}})_
{\U_{\varphi(s+1)},\ldots,\U_{\varphi(k)}}$
for every $1\le s\le k$, for every $\varphi\in\Phi$, and for all $j_1\le\ldots\le j_s\le n$
where $j_t<j_{t+1}$ whenever $\varphi(t)\ge\varphi(t+1)$.
\end{enumerate}

\noindent
\emph{Notation.} For simplicity, in the above formulas we agreed on the following:
\begin{itemize}
\item[--]
When $s=1$, $X^\varphi_{a_{\varphi(1),j_1},\ldots,a_{\varphi(s-1),j_{s-1}}}$
means $X^\varphi$.
\item[--]
When $s=k$, $(X^\varphi_{a_{\varphi(1),j_1},\ldots,a_{\varphi(s-1),j_{s-1}}})_
{\U_{\varphi(s+1)},\ldots,\U_{\varphi(k)}}$ means
\\
$X^\varphi_{a_{\varphi(1),j_1},\ldots,a_{\varphi(k-1),j_{k-1}}}$.
\end{itemize}

Note that above properties will yield the desired result.
Indeed, property $(1)_n$ for every $n$
says that $a_{\ell,n}\ne a_{\ell',n'}$ when $(\ell,n)\ne(\ell',n')$;
and property $(2)_n$ for every $n$ says that $(a_{\ell,n})_n$ is increasing
whenever $I_\ell=\N$. Finally, property $(3)_n$ for every $n$, in the special case $s=k$, says that
$a_{\varphi(k),j_k}\in X^\varphi_{a_{\varphi(1),j_1},\ldots,a_{\varphi(k-1),j_{k-1}}}$,
\emph{i.e.} $(a_{\varphi(1),j_1},\ldots,a_{\varphi(k-1),j_{k-1}},a_{\varphi(k),j_k})\in X^\varphi$,
for every $\varphi\in\Phi$ and for all $j_1\le\ldots\le j_{k}$,
where $j_t<j_{t+1}$ whenever $\varphi(t)\ge\varphi(t+1)$.

\smallskip
Let us start with the base step $n=1$.
We observe that in this case the three properties above reduce to the following:
\begin{itemize}
\item[$(\ast)$]
$a_{\ell,1}\ne a_{\ell',1}$ for all $\ell'<\ell$. 
\item[$(\star)$]
$a_{\varphi(s),1}\in(X^\varphi_{a_{\varphi(1),1},\ldots,a_{\varphi(s-1),1}})_
{\U_{\varphi(s+1)}\otimes\ldots\otimes\U_{\varphi(k)}}$
for every $1\le s\le k$ and for every $\varphi\in\Phi$ where $\varphi(1)<\ldots<\varphi(s)$.
\end{itemize}

We define elements $a_{\ell,1}$ by recursion on $\ell=1,\ldots,m$. 
By hypothesis we know that 
$X^\varphi\in\U_{\varphi(1)}\otimes\U_{\varphi(2)}\otimes\ldots\otimes\U_{\varphi(k)}$, and so:
\begin{itemize}
\item
$(X^\varphi)_{\U_{\varphi(2)},\ldots, \U_{\varphi(k)}}\in\U_{\varphi(1)}$
for every $\varphi\in\Phi$.
\end{itemize}
Pick an element
$$a_{1,1}\in\bigcap
\left\{(X^\varphi)_{\U_{\varphi(2)},\ldots,\U_{\varphi(k)}}\mid 
\varphi\in\Phi,\ \varphi(1)=1\right\}\in\U_1.$$
(As observed in Remark \ref{remark-varphi(1)}, by reordering the sets $I_1,\ldots,I_m$
if necessary, one can always assume that there exists
$\varphi\in\Phi$ with $\varphi(1)=1$.)
So, we have:
\begin{itemize}
\item
$a_{\varphi(1),1}\in(X^\varphi)_{\U_{\varphi(2)},\ldots,\U_{\varphi(k)}}\in\U_{\varphi(1)}$
for all $\varphi\in\Phi$ where $\varphi(1)=1$.
\end{itemize}
In particular:
\begin{itemize}
\item
$X^\varphi_{a_{\varphi(1),1}}\in\U_{\varphi(2)}\otimes\U_{\varphi(3)}\otimes\ldots\otimes\U_{\varphi(k)}$,
and hence 
\\
$(X^\varphi_{a_{\varphi(1),1}})_{\U_{\varphi(3)},\ldots,\U_{\varphi(k)}}\in\U_{\varphi(2)}$
for all $\varphi\in\Phi$ where $\varphi(1)=1$.
\end{itemize}
Then pick an element
\begin{multline*}
a_{2,1}\in\bigcap\left\{(X^\varphi)_{\U_{\varphi(2)},\ldots,\U_{\varphi(k)}}\mid 
\varphi\in\Phi,\ \varphi(1)=2\right\}\ \cap
\\
\bigcap\left\{(X^\varphi_{a_{\varphi(1),1}})_{\U_{\varphi(3)},\ldots,\U_{\varphi(k)}}\mid 
\varphi\in\Phi,\ \varphi(1)=1,\ \varphi(2)=2\right\}\in\U_2,
\end{multline*}
where we agree that the intersection equals $I_2$ in the case where these families are empty.\footnote
{~Note that it is well possible that there are no $\varphi\in\Phi$ such that $\varphi(1)=2$,
and that there are no $\varphi\in\Phi$ such that $\varphi(1)=1$ and $\varphi(2)=2$.
In this case, the element $a_{2,1}$ will play no role, since it will never appear in property $(2 c)$
of the theorem we are proving.}
Then we have: 
\begin{itemize}
\item
$a_{\varphi(1),1}\in(X^\varphi)_{\U_{\varphi(2)},\ldots,\U_{\varphi(k)}}\in\U_{\varphi(1)}$
for every $\varphi\in\Phi$ where $\varphi(1)=2$.
\item
$a_{\varphi(2),1}\in(X^\varphi_{a_{\varphi(1),1}})_{\U_{\varphi(3)},\ldots,\U_{\varphi(k)}}\in\U_{\varphi(2)}$
for every $\varphi\in\Phi$ (if any) where $\varphi(1)=1$ and $\varphi(2)=2$.
\end{itemize}

Inductively iterate the construction, and for $\ell\le m$ pick an element
$$a_{\ell,1}\in\bigcap_{s=1}^\ell
\left\{(X^\varphi_{a_{\varphi(1),1},\ldots,a_{\varphi(s-1),1}})_
{\U_{\varphi(s+1)},\ldots,\U_{\varphi(k)}}\mid 
\varphi\in\Phi,\ \varphi(1)<\ldots<\varphi(s)=\ell\right\}\in\U_\ell,$$
where we agree that the intersection equals $I_\ell$ in the case where these families are empty.
Since $\U_\ell$ is non-principal, we can take $a_{\ell,1}\notin\{a_{1,1},\ldots,a_{\ell-1,1}\}$.
Then for every $s=1,\ldots,\ell$, we have:
\begin{itemize}
\item
$a_{\varphi(s),1}\in (X^\varphi_{a_{\varphi(1),1},\ldots,a_{\varphi(s-1),1}})_
{\U_{\varphi(s+1)},\ldots,\U_{\varphi(k)}}\in\U_{\varphi(s)}$
for every $\varphi\in\Phi$ (if any) where $\varphi(1)<\ldots<\varphi(s)=\ell$.
\end{itemize}
%
It can be verified in a straightforward manner that
the defined elements $a_{\ell,1}$ for $\ell=1,\ldots,m$
satisfy the desired properties $(\ast)$ and $(\star)$.


\smallskip
To better illustrate the idea of the general construction, let us consider
also the next step $n=2$ in detail.
By recursion, we define elements $a_{\ell,2}$ for $\ell=1,\ldots,m$.
We observe that the conditions that must be satisfied
at this step reduce to the following:
\begin{enumerate}
\item[$(1)_{n=2}$]
$a_{\ell,2}\ne a_{\ell',1}$ for all $\ell,\ell'$; and $a_{\ell,2}\ne a_{\ell',2}$ for all $\ell'<\ell$. 
\item[$(2)_{n=2}$]
$a_{\ell,2}>a_{\ell,1}$ whenever $I_\ell=\N$.
\item[$(3)_{n=2}$]
$a_{\varphi(s+t),2}\in(X^\varphi_{a_{\varphi(1),1},\ldots,a_{\varphi(s-1),1},
a_{\varphi(s),2},\ldots,a_{\varphi(s+t-1),2}})_
{\U_{\varphi(s+t+1)},\ldots,\U_{\varphi(k)}}\in\U_{\varphi(s+t)}$
for every $1\le s\le k$, for every $0\le t\le k-s$, and for every $\varphi\in\Phi$ where 
$\varphi(1)<\ldots<\varphi(s-1)$ and
$\varphi(s)<\ldots<\varphi(s+t)$.\footnote
{~When $s=1$ we agree that
$X^\varphi_{a_{\varphi(1),1},\ldots,a_{\varphi(s-1),1},a_{\varphi(s),2},\ldots,a_{\varphi(s+t-1),2}}=
X^\varphi_{a_{\varphi(1),2},\ldots,a_{\varphi(t),2}}$,
and that the condition on $\varphi\in\Phi$ is $\varphi(s)<\ldots<\varphi(s+t)$.
When $t=0$ we agree that 
$X^\varphi_{a_{\varphi(1),1},\ldots,a_{\varphi(s-1),1},a_{\varphi(s),2},\ldots,a_{\varphi(s+t-1),2}}=
X^\varphi_{a_{\varphi(1),1},\ldots,a_{\varphi(s-1),1}}$,
and that the condition on $\varphi\in\Phi$ is $\varphi(1)<\ldots<\varphi(s-1)$.
Finally, when $s=k$ the variable $t$ does not appear, and 
we agree that 
$X^\varphi_{a_{\varphi(1),1},\ldots,a_{\varphi(s-1),1},a_{\varphi(s),2},\ldots,a_{\varphi(s+t-1),2}}=
X^\varphi_{a_{\varphi(1),1},\ldots,a_{\varphi(k-1),1}}$,
and that the condition on $\varphi\in\Phi$ is $\varphi(1)<\ldots<\varphi(k-1)$.}
\end{enumerate}
At the base step $\ell=1$, pick an element
$$a_{1,2}\in\bigcap_{s=1}^{k}\left\{(X^\varphi_{a_{\varphi(1),1},\ldots,a_{\varphi(s-1),1}})_
{\U_{\varphi(s+1)},\ldots,\U_{\varphi(k)}}\mid
\varphi(1)<\ldots<\varphi(s-1), \varphi(s)=1\right\}\in\U_1.$$
Since $\U_1$ is non-principal, we can take $a_{1,2}\notin\{a_{1,1},\ldots,a_{k,1}\}$,
and $a_{1,2}>a_{1,1}$ in case $I_1=\N$.
Then for every $1\le s\le k$ we have:
\begin{itemize}
\item
$a_{\varphi(s),2}\in
(X^\varphi_{a_{\varphi(1),1},\ldots,a_{\varphi(s-1),1}})_
{\U_{\varphi(s+1)},\ldots,\U_{\varphi(k)}}\in\U_{\varphi(s)}$
for every $\varphi\in\Phi$ where $\varphi(1)<\ldots<\varphi(s-1)$ and $\varphi(s)=1$.
\end{itemize}
In particular:
\begin{itemize}
\item
$X^\varphi_{a_{\varphi(1),1},\ldots,a_{\varphi(s-1),1},a_{\varphi(s),2}}\in
\U_{\varphi(s+1)}\otimes\U_{\varphi(s+2)}\otimes\ldots\otimes\U_{\varphi(k)}$, and hence
\\
$(X^\varphi_{a_{\varphi(1),1},\ldots,a_{\varphi(s-1),1},a_{\varphi(s),2}})_
{\U_{\varphi(s+2)},\ldots,\U_{\varphi(k)}}\in\U_{\varphi(s+1)}$ 
for every $\varphi\in\Phi$ where $\varphi(1)<\ldots<\varphi(s-1)$ and $\varphi(s)=1$.
\end{itemize}
At the next step $\ell=2$, pick an element
\begin{align*}
a_{2,2}\in
&
\bigcap_{s=1}^{k}\Big\{(X^\varphi_{a_{\varphi(1),1},\ldots,a_{\varphi(s-1),1}})_
{\U_{\varphi(s+1)}\otimes\ldots\otimes\U_{\varphi(k)}}\,\Big|\,
\\
&\quad\quad\quad\quad\quad
\varphi\in\Phi,\ \varphi(1)<\ldots<\varphi(s-1), \varphi(s)=2\Big\}\ \cap
\\
&
\bigcap_{s=1}^{k-1}\Big\{(X^\varphi_{a_{\varphi(1),1},\ldots,a_{\varphi(s-1),1},a_{\varphi(s),2}})_
{\U_{\varphi(s+2)}\otimes\ldots\otimes\U_{\varphi(k)}}\,\Big|\, 
\\
&\quad\quad\quad\quad\quad
\varphi\in\Phi,\ \varphi(1)<\ldots<\varphi(s-1), \varphi(s)=1, \varphi(s+1)=2\Big\}\in\U_2.
\end{align*}

Since $\U_2$ is non-principal, we can take $a_{2,2}\notin\{a_{1,1},\ldots,a_{k,1},a_{1,2}\}$,
and we can take $a_{2,2}>a_{2,1}$ in case $I_2=\N$.
Then for every $1\le s\le k$ we have: 
\begin{itemize}
\item
$a_{\varphi(s),2}\in
(X^\varphi_{a_{\varphi(1),1},\ldots,a_{\varphi(s-1),1}})_
{\U_{\varphi(s+1)}\otimes\ldots\otimes\U_{\varphi(k)}}\in\U_{\varphi(s)}$
for every $\varphi\in\Phi$ where $\varphi(1)<\ldots<\varphi(s-1)$ and $\varphi(s)=2$;
\end{itemize}
and for every $1\le s\le k-1$ we have:
\begin{itemize}
\item
$a_{\varphi(s+1),2}\in
(X^\varphi_{a_{\varphi(1),1},\ldots,a_{\varphi(s),1}})_
{\U_{\varphi(s+2)}\otimes\ldots\otimes\U_{\varphi(k)}}\in\U_{\varphi(s+1)}$
for every $\varphi\in\Phi$ where $\varphi(1)<\ldots<\varphi(s-1)$ and $\varphi(s)=1$, $\varphi(s+1)=2$.
\end{itemize}
In particular:
\begin{itemize}
\item
$X^\varphi_{a_{\varphi(1),1},\ldots,a_{\varphi(s-1),1},a_{\varphi(s),2}}\in
\U_{\varphi(s+1)}\otimes\U_{\varphi(s+2)}\otimes\ldots\otimes\U_{\varphi(k)}$, and hence
\\
$(X^\varphi_{a_{\varphi(1),1},\ldots,a_{\varphi(s-1),1},a_{\varphi(s),2}})_
{\U_{\varphi(s+2)}\otimes\ldots\otimes\U_{\varphi(k)}}\in\U_{\varphi(s+1)}$ 
for every $\varphi\in\Phi$ where 
\\
$\varphi(1)<\ldots<\varphi(s-1)$ and $\varphi(s)=2$.
\item
$X^\varphi_{a_{\varphi(1),1},\ldots,a_{\varphi(s),1},a_{\varphi(s+1),2}}\in
\U_{\varphi(s+2)}\otimes\U_{\varphi(s+3)}\otimes\ldots\otimes\U_{\varphi(k)}$ and hence
\\
$(X^\varphi_{a_{\varphi(1),1},\ldots,a_{\varphi(s),1},a_{\varphi(s+1),2}})_
{\U_{\varphi(s+3)}\otimes\ldots\otimes\U_{\varphi(k)}}\in\U_{\varphi(s+2)}$ 
for every $\varphi\in\Phi$ where $\varphi(1)<\ldots<\varphi(s-1)$ and $\varphi(s)=1$
and $\varphi(s+1)=2$.
\end{itemize}

At step $\ell=3$, pick an element
\begin{align*}
a_{3,2}\in
&
\bigcap_{s=1}^{k}\Big\{(X^\varphi_{a_{\varphi(1),1},\ldots,a_{\varphi(s-1),1}})_
{\U_{\varphi(s+1)}\otimes\ldots\otimes\U_{\varphi(k)}}\,\Big|\,
\\
& \quad\quad\quad\quad
\varphi\in\Phi,\ \varphi(1)<\ldots<\varphi(s-1), \varphi(s)=3\Big\}\ \cap
\\
& \bigcap_{s=1}^{k-1}\Big\{(X^\varphi_{a_{\varphi(1),1},\ldots,a_{\varphi(s-1),1},a_{\varphi(s),2}})_
{\U_{\varphi(s+2)}\otimes\ldots\otimes\U_{\varphi(k)}}\,\Big|\,
\\
& \quad\quad\quad\quad
\varphi\in\Phi,\ \varphi(1)<\ldots<\varphi(s-1), \varphi(s)=1, \varphi(s+1)=3\Big\}\ \cap
\\
&
\bigcap_{s=1}^{k-1}\Big\{(X^\varphi_{a_{\varphi(1),1},\ldots,a_{\varphi(s-1),1},a_{\varphi(s),2}})_
{\U_{\varphi(s+2)}\otimes\ldots\otimes\U_{\varphi(k)}}\,\Big| 
\\
& \quad\quad\quad\quad
\varphi\in\Phi,\ \varphi(1)<\ldots<\varphi(s-1), \varphi(s)=2, \varphi(s+1)=3\Big\}\ \cap
\\
&
\bigcap_{s=1}^{k-2}\Big\{(X^\varphi_{a_{\varphi(1),1},\ldots,a_{\varphi(s-1),1},a_{\varphi(s),2},a_{\varphi(s+1),2}})_
{\U_{\varphi(s+3)}\otimes\ldots\otimes\U_{\varphi(k)}}\,\Big|
\\
& \quad
\varphi\in\Phi,\ \varphi(1)<\ldots<\varphi(s-1), \varphi(s)=1, \varphi(s+1)=2, \varphi(s+2)=3\Big\}\in\U_3.
\end{align*}
Since $\U_3$ is non-principal, we can take $a_{3,2}\notin\{a_{1,1},\ldots,a_{k,1},a_{1,2}, a_{2,2}\}$.

Iterate the construction, and for $\ell\le k$ pick an element
\begin{multline*}
a_{\ell,2}\in\bigcap_{s=1}^k\bigcap_{t=0}^{k-s}
\Big\{(X^\varphi_{a_{\varphi(1),1},\ldots,a_{\varphi(s-1),1},a_{\varphi(s),2},\ldots,a_{\varphi(s+t-1),2}})_
{\U_{\varphi(s+t+1)}\otimes\ldots\otimes\U_{\varphi(k)}}\mid 
\\
\varphi\in\Phi,\ \varphi(1)<\ldots<\varphi(s-1), \varphi(s)<\ldots<\varphi(s+t-1)<\varphi(s+t)=\ell\Big\}\in\U_\ell.
\end{multline*}
Since $\U_\ell$ is non-principal, we can take $a_{\ell,2}\notin\{a_{1,1},\ldots,a_{k,1},a_{1,2},\ldots,a_{\ell-1,2}\}$.
We observe that, by our definition, for every $s+t\le k$, we have:
\begin{itemize}
\item
$a_{\varphi(s+t),2}\in (X^\varphi_{a_{\varphi(1),1},\ldots,a_{\varphi(s-1),1},a_{\varphi(s),2},\ldots,a_{\varphi(s+t-1),2}})_
{\U_{\varphi(s+t+1)}\otimes\ldots\otimes\U_{\varphi(k)}}\in\U_{\varphi(s+t)}$
for every $\varphi\in\Phi$ where $\varphi(1)<\ldots<\varphi(s-1)$
and $\varphi(s)<\ldots<\varphi(s+t)$.
\end{itemize}
This shows that the defined elements $a_{1,2},\ldots,a_{k,2}$
satisfy the desired properties $(1)_{n=2}$, $(2)_{n=2}$, and $(3)_{n=2}$.

\smallskip
Let us finally turn the general case $n+1>2$.
Assume by inductive hypothesis that properties $(1)_m$, $(2)_m$, and $(3)_m$ hold
for every $m\le n$.

By recursion on $\ell$, we define elements $a_{\ell,n+1}$.
At the base step $\ell=1$, pick an element
\begin{multline*}
a_{1,n+1}\in\bigcap_{s=1}^{k}\Big\{(X^\varphi_{a_{\varphi(1),j_1},\ldots,a_{\varphi(s-1),j_{s-1}}})_
{\U_{\varphi(s+1)}\otimes\ldots\otimes\U_{\varphi(k)}}\mid
\varphi\in\Phi,\ \varphi(s)=1,
\\
 j_1\le\ldots\le j_{s-1}\le n\ \text{where}\ 
j_t<j_{t+1}\ \text{whenever}\ 
\varphi(t)>\varphi(t+1)\Big\}\in\U_1.
\end{multline*}
Since $\U_1$ is non-principal, we can take 
$a_{1,n+1}\notin\{a_{\ell,m}\mid \ell=1,\ldots,k,\ m\le n\}$.
Then for every $1\le s\le k$ we have:
\begin{itemize}
\item
$a_{\varphi(s),n+1}\in
(X^\varphi_{a_{\varphi(1),j_1},\ldots,a_{\varphi(s-1),j_{s-1}}})_
{\U_{\varphi(s+1)}\otimes\ldots\otimes\U_{\varphi(k)}}\in\U_{\varphi(s)}$
for every $\varphi\in\Phi$ such that $\varphi(s)=1$,
and for all $j_1\le\ldots\le j_{s-1}\le n$ where
$j_t<j_{t+1}$ whenever $\varphi(t)>\varphi(t+1)$.
\end{itemize}
In particular:
\begin{itemize}
\item
$X^\varphi_{a_{\varphi(1),j_1},\ldots,a_{\varphi(s-1),j_{s-1}},a_{\varphi(s),n+1}}\in
\U_{\varphi(s+1)}\otimes\U_{\varphi(s+2)}\otimes\ldots\otimes\U_{\varphi(k)}$, and hence
\\
$(X^\varphi_{a_{\varphi(1),j_1},\ldots,a_{\varphi(s-1),j_{s-1}},a_{\varphi(s),n+1}})_
{\U_{\varphi(s+2)}\otimes\ldots\otimes\U_{\varphi(k)}}\in\U_{\varphi(s+1)}$ 
for every $\varphi\in\Phi$ such that $\varphi(s)=1$,
and for all $j_1\le\ldots\le j_{s-1}\le n$ where
$j_t<j_{t+1}$ whenever $\varphi(t)>\varphi(t+1)$.
\end{itemize}
At the next step $\ell=2$, pick an element
\begin{align*}
a_{2,n+1}\in
&
\bigcap_{s=1}^{k}\Big\{(X^\varphi_{a_{\varphi(1),j_1},\ldots,a_{\varphi(s-1),j_{s-1}}})_
{\U_{\varphi(s+1)}\otimes\ldots\otimes\U_{\varphi(k)}}\,\Big|\,\varphi\in\Phi,\,\varphi(s)=2,
\\
&\quad
j_1\le\ldots\le j_{s-1}\le n\ \text{where}\ 
j_t<j_{t+1}\ \text{whenever}\ 
\varphi(t)>\varphi(t+1)\Big\}\ \cap
\\
&
\bigcap_{s=1}^{k-1}\Big\{(X^\varphi_{a_{\varphi(1),j_1},\ldots,a_{\varphi(s-1),j_{s-1}},a_{\varphi(s),n+1}})_
{\U_{\varphi(s+2)}\otimes\ldots\otimes\U_{\varphi(k)}}\,\Big|\,\varphi\in\Phi,\,\varphi(s)=1,
\\
&
\hspace{-1.3cm}
\varphi(s+1)=2,\,
j_1\le\ldots\le j_{s-1}\le n\ \text{where}\ 
j_t<j_{t+1}\ \text{whenever}\ 
\varphi(t)>\varphi(t+1)\Big\}\in\U_2.
\end{align*}
Since $\U_2$ is non-principal, we can take 
$a_{2,n+1}\notin\{a_{\ell,m}\mid \ell=1,\ldots,k,\ m\le n\}$
with $a_{2,n+1}\ne a_{1,n+1}$.
Then for every $1\le s\le k$ we have:
\begin{itemize}
\item
$a_{\varphi(s),n+1}\in
(X^\varphi_{a_{\varphi(1),j_1},\ldots,a_{\varphi(s-1),j_{s-1}}})_
{\U_{\varphi(s+1)}\otimes\ldots\otimes\U_{\varphi(k)}}\in\U_{\varphi(s)}$
for every $\varphi\in\Phi$ such that $\varphi(s)=2$,
and for all $j_1\le\ldots\le j_{s-1}\le n$ where
$j_t<j_{t+1}$ whenever $\varphi(t)>\varphi(t+1)$;
\end{itemize}
and for every $1\le s\le k-1$ we have:
\begin{itemize}
\item
$a_{\varphi(s+1),n+1}\in
(X^\varphi_{a_{\varphi(1),j_1},\ldots,a_{\varphi(s-1),j_{s-1}},a_{\varphi(s),n+1}})_
{\U_{\varphi(s+2)}\otimes\ldots\otimes\U_{\varphi(k)}}\in\U_{\varphi(s+1)}$
for every $\varphi\in\Phi$ such that $\varphi(s)=1$ and $\varphi(s+1)=2$,
and for all $j_1\le\ldots\le j_{s-1}\le n$ where
$j_t<j_{t+1}$ whenever $\varphi(t)>\varphi(t+1)$.
\end{itemize}

Iterate the construction, and for $\ell\le k$ pick an element
\begin{align*}
a_{\ell,n+1}\in
&
\bigcap_{s=1}^k\bigcap_{t=0}^{k-s}
\Big\{(X^\varphi_{a_{\varphi(1),j_1},\ldots,a_{\varphi(s-1),j_{s-1}},a_{\varphi(s),n+1},\ldots,a_{\varphi(s+t-1),n+1}})_
{\U_{\varphi(s+t+1)}\otimes\ldots\otimes\U_{\varphi(k)}}\,\Big|
\\
&\quad
\varphi\in\Phi,\,\varphi(s)<\ldots<\varphi(s+t-1)<\varphi(s+t)=\ell,
\\
&\quad
j_1\le\ldots\le j_{s-1}\le n\ \text{where}\ 
j_t<j_{t+1}\ \text{whenever}\ 
\varphi(t)>\varphi(t+1)\Big\}\in\U_\ell.
\end{align*}
Since $\U_\ell$ is non-principal, we can take 
$$a_{\ell,n+1}\notin\{a_{\ell,m}\mid \ell=1,\ldots,k,\ m\le n\}\cup\{a_{\ell',n+1}\mid\ell'<\ell\}.$$
We observe that, by our definition, for every $s+t\le k$, we have:
$$a_{\varphi(s+t),n+1}\in 
(X^\varphi_{a_{\varphi(1),j_1},\ldots,a_{\varphi(s-1),j_{s-1}},a_{\varphi(s),n+1},\ldots,a_{\varphi(s+t-1),n+1}})_
{\U_{\varphi(s+t+1)}\otimes\ldots\otimes\U_{\varphi(k)}}\in\U_{\varphi(s+t)}$$
for every $\varphi\in\Phi$ where $\varphi(s)<\ldots<\varphi(s+t-1)<\varphi(s+t)=\ell$,
and for all $j_1\le\ldots\le j_{s-1}\le n$ where
$j_t<j_{t+1}$ whenever $\varphi(t)>\varphi(t+1)$.

It can be finally verified that the defined elements $a_{1,n+1},\ldots,a_{k,n+1}$
satisfy the desired properties $(1)_{n+1}$, $(2)_{n+1}$, and $(3)_{n+1}$.

\end{proof}


\begin{thebibliography}{}

\bibitem{acg}
U. Andrews, G. Conant, and I. Goldbring,
Definable sets containing product sets
in expansions of groups, \emph{J. Group Theory} \textbf{22} (2019), 63--82.

\bibitem{ca1}
H. Cartan, Th\'eorie des filtres,
\emph{C. R. Acad. Sci. Paris}, 
\textbf{205} (1937), 595--598.

\bibitem{ca2}
H. Cartan, Filtres et ultrafiltres, 
\emph{C. R. Acad. Sci. Paris}, 
\textbf{205} (1937),  777--779.



\bibitem{ck}
C.C. Chang and H.J. Keisler, \emph{Model Theory}, 3rd edition, North-Holland, 1990.

\bibitem{dn0}
M. Di Nasso, 
Iterated hyper-extensions and an idempotent ultrafilter proof of Rado's Theorem, 
\emph{Proc. Amer. Math. Soc.} \textbf{143} (2015), 1749--1761.

\bibitem{dn}
M. Di Nasso, Hypernatural numbers as ultrafilters, Chapter XI of \cite{lw}.

\bibitem{dgjllm}
M. Di Nasso, I. Goldbring, R. Jin, S. Leth, M. Lupini, and K. Mahlburg,
On a sumset conjecture of Erd\"os, \emph{Canad. J. Math.} \textbf{67} (2015), 795--809.

\bibitem{dj}
M. Di Nasso and R. Jin, 
Foundations of iterated star maps and their use in combinatorics, 
\emph{Ann. Pure Appl. Logic} \textbf{176} (2025).


\bibitem{lw}
P.A. Loeb and M. Wolff (eds.), \emph{Nonstandard Analysis for the Working Mathematician}, 
2nd edition, Springer, 2015.




\bibitem{dgl}
M. Di Nasso, I. Goldbring, and M. Lupini, \emph{Nonstandard Methods in
Ramsey Theory and Combinatorial Number Theory}, Lecture Notes in Mathematics \textbf{2239},
Springer, 2019.

\bibitem{dlmmr}
M. Di Nasso, L. Luperi Baglini, M. Mamino, R. Mennuni, and M. Ragosta, 
Ramsey's witnesses (submitted), arXiv:2503.09246.

\bibitem{dr}
M. Di Nasso and M. Ragosta, Monochromatic exponential triples: an ultrafilter proof,
\emph{Proc. Amer. Math. Soc.}, \textbf{152} (2024), 81--87.
\bibitem{go}
R. Goldblatt, \emph{Lectures on the Hyperreals – An Introduction to Nonstandard Analysis}, 
Graduate Texts in Mathematics \textbf{188}, Springer, 1998.

\bibitem{gb}
I. Goldbring,
Ultrafilter methods in combinatorics, 
\emph{Snapshots of Modern Mathematics from Oberwolfach}, 2021-006.

\bibitem{je}
T. Jech, \emph{Set Theory}, 3rd edition, Springer, 2002.

\bibitem{hl}
N. Hindman and I. Leader, 
Nonconstant monochromatic solutions to systems of linear equations, in
\emph{Topics in Discrete Mathematics}, Springer (2006), 145--154.

\bibitem{hs}
N. Hindman and D. Strauss, \emph{Algebra in the Stone-\v{C}ech
Compactification -- Theory and Applications} (2nd edition),
W. de Gruyter, 2012.

\bibitem{mrr}
J. Moreira, F.K. Richter, and D. Robertson,
A proof of a sumset conjecture of Erd\"os, 
\emph{Annals of Mathematics} \textbf{189} (2019), 605--652.

\bibitem{pu}
C. Puritz, Ultrafilters and standard functions in nonstandard analysis,
\emph{Proc. Lond. Math. Soc.} \textbf{22} (1971), 706--733.

\end{thebibliography}
\end{document}